\documentclass[12pt,leqno]{article}

\usepackage{amsfonts,amsthm,amssymb,amsmath,graphicx}

\newcommand{\fe}{\mathfrak {e}}

\newcommand{\fg}{\mathfrak {g}}

\newcommand{\fk}{\mathfrak {k}}

\newcommand{\fl}{\mathfrak {l}}
\newcommand{\fp}{\mathfrak {p}}

\newcommand{\fo}{\mathfrak {o}}
\newcommand{\fs}{\mathfrak {s}}
\newcommand{\fu}{\mathfrak {u}}

\newcommand{\fz}{\mathfrak {z}}

\newcommand{\dyr}{{\rm D}^r}
\newcommand{\si}{\sigma}

\newcommand{\thq}{\theta}

\newcommand{\sk}{\bigskip}

\newcommand{\lra}{\longrightarrow}

\newcommand{\KK}{\De_1^d}

\newcommand{\dy}{{\rm D}}
\newcommand{\dyo}{{\rm D}^1}
\newcommand{\dyt}{{\rm D}^2}

\newcommand{\co}{{\cal O}}

\newcommand{\De}{\Delta}

\newcommand{\al}{\alpha}
\newcommand{\be}{\beta}
\newcommand{\de}{\delta}

\newcommand{\ga}{\gamma}

\newcommand{\vf}{\varphi}

\newcommand{\we}{\wedge}

\newcommand{\bsl}{\backslash}

\newcommand{\aut}{{\rm Aut}}
\newcommand{\inv}{{\rm Inv}}

\newcommand{\om}{\omega}

\newcommand{\Th}{\Theta}

\newcommand{\ad}{{\rm ad}}

\newtheorem{proposition}{Proposition}[section]
\newtheorem{theorem}[proposition]{Theorem}
\newtheorem{corollary}[proposition]{Corollary}
\newtheorem{lemma}[proposition]{Lemma}
\newtheorem{definition}[proposition]{Definition}

\newtheorem{remark}[proposition]{Remark}

\newcommand{\bz}{\mathbb {Z}}
\newcommand{\bc}{\mathbb {C}}

\newcommand{\br}{\mathbb {R}}
\newcommand{\bh}{\mathbb {H}}

\begin{document}

\thispagestyle{empty}

\begin{center}
{\Large \sc Affine Vogan Diagrams
and Symmetric Pairs}
\end{center}

\sk

\sk

{\large Meng-Kiat Chuah

\sk

Department of Mathematics,

National Tsing Hua University,

Hsinchu, Taiwan.

\sk

{\tt chuah@math.nthu.edu.tw}}

\sk

\sk

\sk

\sk

\sk

\sk

\noindent {\bf Abstract: }

We introduce the affine Vogan diagrams
of complex simple Lie algebras.
These are generalizations of Vogan diagrams,
and we study the involutions represented by them.
We apply these diagrams to study the symmetric pairs,
in particular the
associated and symplectic symmetric pairs.

\sk

\sk


\sk

\sk



\newpage
\section{Introduction}
\setcounter{equation}{0}

Let $L$ be a finite dimensional complex simple Lie algebra
with Dynkin diagram $\dy$.
A Vogan diagram on $\dy$ \cite[VI-8]{kn} represents
an $L$-involution, and conversely an $L$-involution
can be represented by a Vogan diagram on $\dy$.
Let $\dyr$ be the affine Dynkin diagram of $L$ \cite[Ch.4]{ka90}.
We similarly introduce the {\it affine Vogan diagram} on $\dyr$.
Here $\dyo$ is the extended Dynkin diagram,
and we study the affine Vogan diagrams which are not extensions of
Vogan diagrams on $\dy$.
We show that they represent
$L$-involutions, and furthermore we determine these involutions.
We then provide two applications of affine Vogan diagrams,
namely the studies of associated symmetric pairs
and symplectic symmetric pairs.
We now describe our projects in more details.

The affine Dynkin diagrams include $\dyo$ for all complex
simple Lie algebras, $\dyt$ for $A_n, D_n, E_6$,
and $\dy^3$ for $D_4$.
We ignore $\dy^3$ and
consider $\dyr$ for $r=1,2$.
Let $\inv(L)$ denote the $L$-involutions,
namely $\bc$-linear $L$-automorphisms $\si$
such that $\si^2=1$.
A convenient way to describe $\inv(L)$ is given by the
Kac diagrams, or by their invariant subalgebras
(see Definition \ref{dkd} and Theorem \ref{kd}).
In Definition \ref{dvd}, we define affine Vogan diagrams on $\dyr$
and discuss how they may represent $\inv(L)$.
In this definition, we use diagram involutions and circlings
(i.e. a vertex is circled or not) instead of
paintings (i.e. a vertex has white or black color)
because we have used the paintings for Kac diagrams.

We focus on $\dyo$ for now.
A Vogan diagram on $\dy$ always extends to
an affine Vogan diagram on $\dyo$.
But an affine Vogan diagram may not be an extension
of a Vogan diagram, and in that case it is less
clear which involution it represents.
The next theorem solves this problem.
We always regard involutions which are conjugated by
$L$-automorphisms as the same.
We say that two affine Vogan diagrams are
equivalent if they represent the same involution.
Equivalent diagrams are classified in \cite{ch2}.

\begin{theorem}
$\,$

\noindent {\rm (a)} \,
Given a Vogan diagram
on $\dy$ which represents $\si \in \inv(L)$,
it has two extensions to affine Vogan diagrams on
$\dyo = \dy \cup \{\vf\}$, where $\vf$ is fixed by
the diagram involution.
One extension represents $\si$, and the other does not
represent an involution.

\noindent {\rm (b)} \,
Every affine Vogan diagram on $\dyo$
which is not an extension of a
Vogan diagram represents an $L$-involution $\si$.
It is equivalent to an affine Vogan diagram
in the first column of Figure 1,
where $\si$ is revealed by the accompanying
Kac diagram and invariant subalgebra.
\label{thm1}
\end{theorem}

We next discuss some applications of affine Vogan diagrams.
Let $H$ be a $\si$-stable Cartan subalgebra of $L$
with root space decomposition $L= H + \sum_\De L_\al$.
Then $\si$ permutes the root spaces.
The vertices of $\dy$ (resp. $\dyo$)
represent some $\Pi \subset \De$
(resp. $\Pi^1 = \Pi \cup \{\vf\}$),
where $\Pi$ is a simple system with lowest root $\vf$.
We say that $\si$ stabilizes $\Pi$ (resp. $\Pi^1$)
if it stabilizes $\sum_\Pi L_\al$ (resp. $\sum_{\Pi^1} L_\al$),
and it leads to a Vogan diagram on $\dy$
(resp. affine Vogan diagram on $\dyo$).
Given $\si$, there always exist
$\si$-stable $H$ and $\Pi$,
and they imply that $\si$ also stabilizes $\Pi^1$.
So there is apparently no reason
to deal with ``inferior'' Cartan subalgebras
where $\si$ stabilizes merely
$\Pi^1$ but not $\Pi$.
However, such Cartan subalgebras are inevitable
when there are {\it two} involutions,
which arise in the context of symmetric pairs.

A {\it symmetric pair} of $L$ is a pair
$(\fg, \fg^\si)$, where $\fg$ is a real form of $L$,
$\si \in \inv(\fg)$, and $\fg^\si$ is its invariant subalgebra.
They are classified by Berger \cite{be}.
There exists a Cartan involution $\thq$ of $\fg$
which commutes with $\si$. By $\bc$-linear extensions,
$\thq$ and $\si$ extend to $L$.
This leads to a correspondence
\begin{equation}
 \{\mbox{symmetric pairs } (\fg, \fg^\si)\}
\; \longleftrightarrow \;
\{\mbox{commuting $L$-involutions } (\thq, \si)\} ,
\label{idsp}
\end{equation}
where the latter are studied by Helminck \cite{helm}.

A {\it double Vogan diagram} on $\dyr$ is the superimpose of
a Kac diagram and an affine Vogan diagram (Definition \ref{suim}).
Every symmetric pair can be represented by
a double Vogan diagram on $\dyr$ \cite[Thm.1.3]{ch}.
If $r=1$,
there exists a Cartan subalgebra $H$ such that
$\thq|_H =1$, $\si H = H$,
and $\si$ stabilizes $\Pi^1$.
We may not have $\si$-stable $\Pi$,
so the affine Vogan diagram
may not be an extension of a Vogan diagram on $\dy$.

Let $\thq$ and $\si$ commute as above.
Two symmetric pairs $(\fg,\fg^\si)$ and $(\fg,\fg^\tau)$
are said to be {\it associated} to each other if $\si \thq = \tau$.
In other words if $\fg = \fk + \fp$ is a Cartan decomposition
with respect to $\thq$, then
$\si|_\fk = \tau|_\fk$ and $\si|_\fp = -\tau|_\fp$.

We can classify the symmetric pairs by double Vogan diagrams,
as carried out in \cite[\S7]{ch}.
However, $(\fg, \fg^\si)$
may be mismatched with the double Vogan diagram
of its associate symmetric pair
$(\fg, \fg^{\si \thq})$, especially when
the affine Vogan diagram
is not an extension of a Vogan diagram.
Such errors happen in \cite{ch}, and
we provide their corrections with the help of Theorem \ref{thm1}.

\begin{theorem}
{\rm (Corrigendum to \cite[pp.1741,1742,1747]{ch})}
The following symmetric pairs are represented by the double
Vogan diagrams in Figure 2:
\begin{equation}
\begin{array}{cl}
\mbox{\rm (a)} & (\fs \fp(m,m), \fs \fp(m, \bc)) \;,\;
(\fs \fp(m,m), \fg \fl(m, \bh)) ,\\
\mbox{\rm (b)} & (\fs \fo(n,n), \fg \fl(n, \br)) \;,\;
(\fs \fo(n,n), \fs \fo(n, \bc)) ,\; n \mbox{ even}, \\
\mbox{\rm (c)} & (\fs \fo(n,n), \fg \fl(n, \br)) \;,\;
(\fs \fo(n,n), \fs \fo(n, \bc)) ,\; n \mbox{ odd}.
\end{array}
\label{rfi}
\end{equation}
\label{thmcorr}
\end{theorem}

Another application of Theorem \ref{thm1} is the classification
of symplectic symmetric pairs.
The symmetric pair $(\fg,\fg^\si)$ is the infinitesimal version of
symmetric space $G/G^\si$, where $G$ is a connected Lie group
whose Lie algebra is $\fg$.
If $G/G^\si$ has a $G$-invariant symplectic form,
we say that $(\fg,\fg^\si)$ is {\it symplectic}.
If in addition $G/G^\si$ has a $G$-invariant complex structure
such that the symplectic form is pseudo-Hermitian,
we say that $(\fg,\fg^\si)$ is {\it pseudo-Hermitian}.
The pseudo-Hermitian symmetric pairs are classified by Shapiro
\cite{sh}, and the remaining symplectic symmetric pairs
are classified by Bieliavsky \cite{bi}.
The works of \cite{sh}\cite{bi} are done by algebraic
computations of $\fg^\si$,
and we intend to provide an alternative combinatorial classification
by diagrams.
This has been done for the pseudo-Hermitian symmetric pairs \cite{mrl},
and here we do the same for the remaining
symplectic symmetric pairs.

We denote a double Vogan diagram on $\dyr$ by $(p,c,d)$, where
\[
\begin{array}{cl}
p: & \mbox{Kac diagram, namely vertices of $\dyr$ are painted as
white or black,}\\
(c,d): & \mbox{affine Vogan diagram,
namely $d$ is a diagram involution on $\dyr$,}\\
& \mbox{each vertex fixed by $d$ is circled or not circled by $c$.}
\end{array}
\]

In (\ref{cana}), we construct a diagram mapping
$\pi : \dy \lra \dyt$ such that
$\dyt = \pi(\dy) \cup \{\vf\}$.
If $c$ is a circling on $\dyt$,
we define a circling $\pi^{-1}(c)$ on $\dy$
by letting $\al \in \dy$ be circled if and only if $\pi(\al)$
is circled.
Let $c(\dyt)$ denote the circled vertices.
The vertices $\al \in \dyr$ are equipped with positive integers
$m_\al$ such that $\sum_{\dyr} m_\al \al =0$ \cite[Ch.4]{ka90}.

\begin{theorem}
A symmetric pair $(\fg, \fg^\si)$ is symplectic non-pseudo-Hermitian
if and only if it can be represented by a double Vogan diagram
$(p,c,d)$ on $\dyr$ such that:

\noindent {\rm (a)} $\,$ For $r=1$, $(c,d)$
is equivalent to a Kac diagram on $\dyo$ with two black vertices,
and is not an extension of a Vogan diagram on $\dy$;

\noindent {\rm (b)} $\,$ For $r=2$,
either $c(\dyt) = \emptyset$ and $d \neq 1$, or
$\sum_{c(\dyt)} m_x$ is even and
$\pi^{-1}(c) = \{\al, \be\}$
with $m_\al = m_\be =1$ and
$\pi(\al)=\pi(\be)$.
\label{thm2}
\end{theorem}

\begin{corollary}
The symplectic non-pseudo-Hermitian symmetric pairs are classified by
the double Vogan diagrams with $r=1$ in
Figure 3, and $r=2$ in Figure 4.
\label{fma}
\end{corollary}

In Figures 3 and 4, we adopt Cartan's notation
and denote the exceptional real forms by their characters,
as explained in (\ref{chara}).

Figures 3 and 4 are consistent with
Bieliavsky's classification \cite[p.268-269]{bi},
except for a minor error indicated in Remark \ref{bisa}.
Together with the pseudo-Hermitian ones in \cite[Figs.4,5]{mrl},
we obtain an independent diagrammatic classification
of symplectic symmetric pairs.
The classical symplectic
symmetric pairs can often be studied by matrix computations
on the center of $\fg^\si$
(see Definition \ref{lad}).
But for the exceptional cases, the computations
of $\fg^\si$
are more difficult, and hence our diagrams provide a useful alternative
to study such symmetric pairs. They also explicitly reveal
the behaviors of roots under $\thq$ and $\si$.

We remark that affine Vogan diagrams on $\dyr$
have been discussed
in \cite{ba}\cite{ba2}\cite{pa}, but with different meanings.
Let $\dyo = \dy \cup \{\vf\}$, where $\dy$ represents a
simple system $\Pi$ of $L$.
In this article,
$\vf$ represents the lowest root of $\Pi$, so
$(\Pi \cup \{\vf\}) \subset \De \subset H^*$
are linearly dependent.
Hence a circling on $\dyo$ is ``over-determined'',
namely it represents an $L$-involution only if an additional
condition is satisfied (Theorem \ref{vd}(b)).
On the other hand, the vertex $\vf$ has a different meaning in
the affine Kac-Moody
algebras $\widetilde{L}$ \cite{ba}\cite{ba2}.
Their Cartan subalgebras
$\widetilde{H} + \bc c + \bc d$ carry
root systems $\widetilde{\De}$
of the form $\{\al + n \de \;;\; \al \in \De , n \in \bz\}$
and $\{n \de \;;\; 0 \neq n \in \bz\}$,
where $\de$ is an imaginary root \cite[Ch.6]{ka90}.
In that setting, $\vf$ represents
$\mbox{lowest}(\Pi) + \de$, so
$(\Pi \cup \{\vf\}) \subset \widetilde{\De} \subset \widetilde{H}^*$
is a simple system of $\widetilde{L}$.
Then affine Vogan diagrams on $\dyo$ represent involutions on
$\widetilde{L}$. Similarly, \cite{pa} studies
affine Vogan diagrams on $\dyt$ to obtain involutions on the
twisted affine Kac-Moody algebras.

The sections in this article are arranged as follows.
In Section 2, we recall some properties of Kac diagrams
and Vogan diagrams, and introduce the affine Vogan diagrams.
In Section 3, we study some involutions on Lie algebras,
especially the dimensions of their invariant subalgebras.
In Section 4, we find the invariant subalgebras
and prove Theorem \ref{thm1}.
In Section 5, we study associate symmetric pairs and
prove Theorem \ref{thmcorr}.
In Section 6, we study symplectic symmetric pairs,
then prove Theorem \ref{thm2} and Corollary \ref{fma}.

\sk

\noindent {\bf Acknowledgements. }
The author thanks Pierre Bieliavsky
for a helpful discussion.
This work is supported in part
by the Ministry of Science and Technology and
the National Center for Theoretical Sciences of Taiwan.


\newpage
\section{Affine Vogan Diagrams}
\setcounter{equation}{0}

Let $L$ be a complex simple Lie algebra.
In this section, we review the Kac and Vogan diagrams, which are
useful methods to represent $L$-involutions.
We also introduce the affine Vogan diagrams,
which are generalizations of Vogan diagrams.

Let $H$ be a Cartan subalgebra of $L$, with root system
$\De \subset H^*$ and root space decomposition $L = H + \sum_\De L_\al$.
A simple system of $\De$ leads to the Dynkin diagram $\dy$.
Let $\Th$ be a diagram automorphism on $\dy$
of order $r$.
We form the quotient diagram of $\dy$,
where the vertices are $\Th$-orbits of vertices.
For $r=2$, the edges of the quotient diagram follow the
root lengths
\[
\begin{array}{rcl}
\{\mbox{$\Th$-orbits of 2 adjacent vertices}\}
& < & \{\mbox{$\Th$-orbits of 2 non-adjacent vertices}\} \\
& < & \{\mbox{vertices fixed by $\Th$}\}.
\end{array}
\]
Let $\dyr$ consist of the quotient diagram together
with an extra vertex $\vf$ for some lowest weight
\cite[Ch.4]{ka90}. Let
\begin{equation}
 \pi : \dy \lra \dyr
 \label{cana}
 \end{equation}
send each vertex of $\dy$ to its $\Th$-orbit,
so that $\dyr = \pi(\dy) \cup \{\vf\}$.

If $r =1$, then $\dyo = \dy \cup \{\vf\}$
and $\pi$ is the inclusion,
where $\vf$ is the lowest root of $\dy$.
For $r=2$, we illustrate $\pi$ in Figure 6,
where each $\al \in \dy$ is mapped to
$\pi(\al) \in \dyt$ located exactly below $\al$.

There are positive integers
$\{m_\al\}_{\dyr}$ without nontrivial common factor
such that $\sum_{\dyr} m_\al \al =0$
(see Tables Aff 1 and Aff 2 in \cite[Ch.4]{ka90}).
A painting on a diagram assigns white or black color to each vertex.

\begin{definition}
{\rm A} Kac diagram {\rm is a painting on $\dyr$
such that $r \sum_{\rm black} m_\al = 2$.}
\label{dkd}
\end{definition}

If $\si$ is an $L$-involution,
write $L = L^\si + L^{-\si}$ for its $\pm 1$-eigenspaces.
Then $L^\si$ is a reductive Lie algebra acting
on $L^{-\si}$ by the adjoint map. Write
$L^\si = [L^\si, L^\si] + \fz(L^\si)$ to denote
its semisimple part and center.

\begin{theorem}
{\rm \cite[Ch.8]{ka90}}
A Kac diagram represents a unique $L$-involution $\si$
by:

\noindent {\rm (a)} $\;$
The white vertices form the Dynkin diagram of
$[L^\si, L^\si]$,

\noindent {\rm (b)} $\;$
The number of black vertices is $1 + \dim \fz(L^\si)$,

\noindent {\rm (c)} $\;$
The black vertices are the lowest weights of the
$L^\si$-action on $L^{-\si}$.

The Kac diagrams are in bijective
correspondence with the $L$-involutions.
\label{kd}
\end{theorem}

As mentioned in the introduction, we do not distinguish
the $L$-involutions
$\si$ and $x \si x^{-1}$. In Theorem \ref{kd},
we also do not distinguish
a Kac diagram with its image under a diagram symmetry.
The Kac diagrams are listed in \cite[Figs.1-3]{c1}.
Actually Kac's theorem handles all $L$-automorphisms of
finite order, but we need it only for order 2 here.

If $\Pi$ is a simple system, we let $\Pi^1$ be the
union of $\Pi$ and its lowest root.

\begin{definition}
$\,$

{\rm
\noindent (a) $\,$
A {\it Vogan diagram on $\dy$}
(resp. {\it affine Vogan diagram on $\dyr$})
is a pair $(c,d)$,
where $d \in \inv(\dy)$ (resp. $d \in \inv(\dyr)$)
and $c$ is a circling on the vertices fixed by $d$.

\noindent (b) $\,$ For $r=1$,
we say that it represents $\si \in \inv(L)$
if there exists a
$\si$-stable Cartan subalgebra $H$ with a simple system
$\Pi$ such that the vertices of $\dy$ (resp. $\dyo$)
represent $\Pi$ (resp. $\Pi^1$),
and for all $\al \in \dy$ (resp. $\al \in \dyo$),
\begin{equation}
\begin{array}{l}
\si L_\al = L_{d \al} ,\\
d \al = \al \mbox{ and } \al \mbox{ is not circled}
\; \Longleftrightarrow \; \si = 1 \mbox{ on } L_\al ,\\
d \al = \al \mbox{ and } \al \mbox{ is circled}
\; \Longleftrightarrow \; \si = -1 \mbox{ on } L_\al .
\end{array}
\label{dew}
\end{equation}}
\label{dvd}
\end{definition}

In Definition \ref{dvd}(b),
we consider only $\dyo$.
The relations between affine Vogan diagrams on $\dyt$
and involutions are more delicate because
some $x \in \dyt$ satisfy $\dim L_x =2$.
We will discuss this issue in
Definition \ref{suim} and Proposition \ref{jass} later.

From now on through Section 4, we
set $r=1$ and focus on $\dyo = \dy \cup \{\vf\}$
to prove Theorem \ref{thm1}.
Given an affine Vogan diagram $(c,d)$ on $\dyo$,
we let $\co$ be a collection of $d$-orbits given by
\[
\begin{array}{rl}
\co = & \{\al \in \dyo \;;\; \al = d \al \mbox{ and }
\al \mbox{ is circled}\} \\
& \cup \{\{\al, d \al\} \subset \dyo \;;\;
\mbox{$\al$ and $d\al$ are adjacent vertices}\}.
\end{array}
\]
If $x = \{\al, d\al\} \in \co$ belongs to the latter type,
then $m_\al = m_{d \al}$, so we can define $m_x$ accordingly.

\begin{theorem}
{\rm \cite[VI-8]{kn}\cite[Def.1.2,Thm.1.3]{ch}}

\noindent {\rm (a)} $\,$
A Vogan diagram on $\dy$ always represents a unique $L$-involution.
Conversely, an $L$-involution is represented by a Vogan diagram
on $\dy$.

\noindent {\rm (b)} $\,$
An affine Vogan diagram on $\dyo$ represents an $L$-involution
if and only if $\sum_{x \in \co} m_x$ is even.
\label{vd}
\end{theorem}

The next proposition basically says that
affine Vogan diagrams on $\dyo$
are more general than Vogan diagrams on $\dy$.

\begin{proposition}
Given a Vogan diagram $(c,d)$ on $\dy$ which represents $\si$,
it has two extensions to affine Vogan diagrams on
$\dyo = \dy \cup \{\vf\}$, where $\vf$ is fixed by $d$.
One extension represents $\si$, and the other does not
represent an involution.
\label{exx}
\end{proposition}
\begin{proof}
Let $(c,d)$ be a Vogan diagram on $\dy$ which represents $\si$.
There exist $H$ and $\Pi$ such that (\ref{dew}) holds.
Extend $d$ naturally to $\inv(\De)$
by $d(\al + \be)= d \al + d \be$.
Let $\vf = -\sum_\Pi m_\al \al$ be the lowest root of $\Pi$.
We have $m_\al = m_{d \al}$ for all $\al \in \Pi$, so
\[ d \vf = -\sum_\Pi m_\al d \al = -\sum_\Pi m_{d\al} d\al
= -\sum_\Pi m_\al \al = \vf .\]
It implies that $\si L_\vf = L_\vf$.
Extend $c$ to a circling $c_1$ on $\dyo$ by
\begin{equation}
\vf \mbox{ circled} \; \Longleftrightarrow \;
 \si = -1 \mbox{ on } L_\vf .
 \label{hhb}
 \end{equation}
Then (\ref{dew}) holds for $\vf$ as well,
so the affine Vogan diagram $(c_1,d)$
represents $\si$.

Another extension $(c_2,d)$
is obtained by treating $\vf$ oppositely from (\ref{hhb}).
We have $m_\vf =1$, so $(c_1,d)$ and $(c_2,d)$
have opposite parities of $\sum_{x \in \co} m_x$.
Since $(c_1,d)$ represents $\si$, by Theorem \ref{vd}(b),
$(c_2,d)$ does not represent an involution.
\end{proof}

In view of Proposition \ref{exx}, it remains only to consider
the affine Vogan diagrams on $\dyo$ which are
not extensions of Vogan diagrams.
Recall that two affine Vogan diagrams are said to be equivalent
if they represent the same involution.

\begin{proposition}
Every affine Vogan diagram on $\dyo$ which is not an extension of a
Vogan diagram is equivalent to a diagram in
the first column of Figure 1.
It represents an $L$-involution.
\label{licy}
\end{proposition}
\begin{proof}
The affine Vogan diagrams which are
not extensions of Vogan diagrams
carry $d \in \inv(\dyo)$ such that for any
partition $\dyo = \dy \cup \{\vf\}$
where $\vf$ is the lowest root of $\dy$,
we have $d \vf \neq \vf$.
For example this cannot occur on $E_6$
because for any $d \in \inv(\dyo)$,
we can always declare the vertex $\vf$ such that $d \vf = \vf$
and $m_\vf =1$ to be the lowest root of the remaining
vertices.

Given an affine Vogan diagram $(c,d)$ on $\dyo$,
we let it be equivalent to
another diagram $(c',d)$ with as few circled vertices
as possible.
Then $c'$ has at most two circled vertices
\cite[Thm.6.96]{kn}\cite[Table 1]{ch2}.
For all cases where $c'$ has two circled vertices,
we have $d \vf = \vf$, namely it is an extension
of a Vogan diagram on $\dy$.
Therefore, it suffices to consider the cases
with at most one circled vertex
and $d \vf \neq \vf$.
Such diagrams are exhausted by
the first column of Figure 1.
This proves the first statement
of the proposition.

Every affine Vogan diagram in the first column of Figure 1
satisfies $\sum_{x \in \co} m_x$ is even, so
by Theorem \ref{vd}(b), it represents
an involution.
This proves the second statement of the proposition.
\end{proof}


\newpage
\section{Dimensions of Invariant Subalgebras}
\setcounter{equation}{0}

We consider the affine Vogan diagrams on $\dyo$ which are
not extensions of Vogan diagrams.
By Proposition \ref{licy},
they are given by the first column of Figure 1,
and each of them represents an $L$-involution $\si$.
In order to find $\si$, we first compute $\dim L^\si$
for several cases in this section.

Let $(c,d)$ be an affine Vogan diagram on $\dyo$
which represents an involution $\si$.
By Definition \ref{dvd}, there exists a
$\si$-stable Cartan subalgebra $H$
with $\Pi^1 = \Pi \cup \{\vf\}$
such that (\ref{dew}) holds for all $\al \in \dyo$.
Let $\De \subset H^*$ be the root system,
and let $\De^\pm$ be the positive and negative roots
determined by $\Pi$.
Here $d$ permutes $\Pi^1$, and we extend it to
$d \in \inv(\De)$ by $d(\al+\be) = d \al + d \be$. Let
\begin{equation}
\De^d = \{ \al \in \De \;;\; \al = d \al\} .
 \label{ink}
\end{equation}
If $\al \in \De^d$, then $\si$ stabilizes $L_\al$,
so $\si$ acts as $\pm 1$ on $L_\al$. Let
\begin{equation}
\De_{\pm 1}^d = \{\al \in \De^d \;;\;
\si = \pm 1 \mbox{ on } L_\al\} .
\label{ksi}
\end{equation}
Note that $\al \in \De_{\pm 1}^d$
if and only if $-\al \in \De_{\pm 1}^d$
\cite[Prop.2.6(b)]{c1}.
We shall make use of the partition
$\De = \De_1^d \cup \De_{-1}^d \cup (\De \bsl \De^d)$
to compute $\dim L^\si$.

Let $\dim H = n$.
We index the vertices of $\dyo$ by
$\al_0, ..., \al_n$ as indicated in Figure 5.
Some affine Vogan diagrams in Figure 1 satisfy the following.
 \begin{equation}
\begin{tabular}{|c|c|c|} \hline
& & some elements \\
Figure 1 & $\De^d \cap \De^+$ & of $\De^+ \bsl \De^d$ \\ \hline
(b), (f), (i)
&  $m = \frac{n-1}{2}$,
$\{\sum_{j=m-i}^{m+i+1} \al_j \;;\; i=0, 1, ..., m-1\}$
& $\al_{m+1}, ..., \al_n$ \\
(g), (h), (j), (k)
& $m = \frac{n}{2}$,
$\{\sum_{j=m-i}^{m+i} \al_j \;;\; i=0, 1, ..., m-1\}$
& $\al_{m+1}, ..., \al_n$
\\ \hline
\end{tabular}
\label{taw}
\end{equation}

 For instance in Figure 1(b),
 $d \al_{m-i} = \al_{m+i+1}$. Hence
 $d(\al_{m-i} + ... + \al_{m+i+1}) = \al_{m+i+1} + ... + \al_{m-i}$,
 which implies that
 $\al_{m-i} + ... + \al_{m+i+1} \in \De^d \cap \De^+$.

Let $H = H^\si + H^{-\si}$ be the sum of $\pm 1$-eigenspaces.
Let $X_\al \in L_\al$ denote generic root vectors.
A desirable property is
$H^\si = \sum_{\De^d \cap \De^+} \bc [X_\al, X_{-\al}]$,
but it is not satisfied by Figure 1(a)
because $\De^d \cap \De^+ = \emptyset$.
So we disregard Figure 1(a) for the rest of this section
and work on the remaining diagrams.
Let $\langle \De^d \rangle \subset H^*$ denote the subspace
spanned by $\De^d$.

\begin{proposition}
Consider $(c,d)$ in the first column of Figure 1,
but not Figure 1(a).

\noindent {\rm (a)} \,
$H^\si = \sum_{\De^d \cap \De^+} \bc[X_\al, X_{-\al}]$.

\noindent {\rm (b)} \, $\dim L^\si = \dim \langle \De^d \rangle +
|\De_1^d| + \frac{1}{2} |\De \bsl \De^d|$.

\noindent {\rm (c)} \, If the roots in $\De^d \cap \De^+$
are linearly independent, then
$\dim L^\si = |\De_1^d| + \frac{1}{2} |\De|$.
\label{lmao}
\end{proposition}
\begin{proof}
We first prove part (a).
Since $H$ is generated by
$\{[X_\al, X_{-\al}] \;;\; \al \in \De^+\}$,
\begin{equation}
H^\si = \sum_{\De^d \cap \De^+} \bc[X_\al, X_{-\al}] +
\sum_{\De \bsl \De^d} \bc([X_\al, X_{-\al}] + \si [X_\al, X_{-\al}]).
\label{eone}
\end{equation}

We check the diagrams case-by-case and see that (\ref{eone})
leads to part (a).
For example for Figure 1(b),
the middle column of (\ref{taw}) says that $\dim H^\si \geq m$,
and the right column of (\ref{taw}) says that $\dim H^{-\si} \geq n-m$
because each $\al \in \De^+ \bsl \De^d$ leads to
$[X_\al, X_{-\al}] - \si [X_\al, X_{-\al}] \in H^{-\si}$.
Hence $\dim H^\si = m$ and
$H^\si = \sum_{\De^d \cap \De^+} \bc[X_\al, X_{-\al}]$.

Similarly for Figure 1(c),
we have $\dim H^\si \geq n-1$ because it contains
$[X_\al, X_{-\al}]$ for $\al \in \{\al_2, ..., \al_n\}$.
But $H^{-\si} \neq 0$,
hence $\dim H^\si = n-1$ and
$H^\si = \sum_{\De^d \cap \De^+} \bc[X_\al, X_{-\al}]$.
The same happens for all diagrams other than
 Figure 1(a). Part (a) follows.

 Next we prove part (b).
 Since $\si$ stabilizes $H$ and permutes its root spaces, we have
$L^\si = H^\si + (\sum_\De L_\al)^\si$. Together with part (a),
we have
\begin{equation}
\dim L^\si = \dim \langle \De^d \rangle + \dim (\sum_\De L_\al)^\si.
\label{eem}
\end{equation}

To study $\dim (\sum_\De L_\al)^\si$, we consider
each component of $\De = \De_1^d \cup \De_{-1}^d \cup (\De \bsl \De^d)$.
We have
\begin{equation}
\begin{array}{l}
\al \in \De_1^d \Longrightarrow L_\al \subset L^\si , \\
\al \in \De_{-1}^d \Longrightarrow L_\al \cap L^\si = 0 ,\\
\al \in \De \bsl \De^d \Longrightarrow
\dim ((L_\al + \si L_\al) \cap L^\si) = 1 .
\end{array}
\label{eli}
\end{equation}
The last condition of (\ref{eli}) is due to the fact that if
$\al \in \De \bsl \De^d$, then
$X_\al \pm \si X_\al \in L^{\pm \si}$. By (\ref{eli}),
$\dim (\sum_\De L_\al)^\si = |\De_1^d| + \frac{1}{2} |\De \bsl \De^d|$.
Together with (\ref{eem}), this proves part (b).

Finally we prove part (c).
If the roots in $\De^d \cap \De^+$ are linearly independent, then
\begin{equation}
\dim \langle \De^d \rangle = \frac{1}{2} |\De^d| .
\label{een}
\end{equation}
It follows that
\[
\begin{array}{rll}
\dim L^\si &
= \dim \langle \De^d \rangle +
|\De_1^d| + \frac{1}{2} |\De \bsl \De^d| & \mbox{by part (b)} \\
& = \frac{1}{2} |\De^d| +
|\De_1^d| + \frac{1}{2} |\De \bsl \De^d| & \mbox{by (\ref{een})} \\
& = |\De_1^d| + \frac{1}{2} |\De| . &
\end{array}
\]
This proves part (c).
\end{proof}

We shall use Proposition \ref{lmao} to compute $\dim L^\si$ for
several diagrams.
By observing the circlings in Figure 1,
the middle column of (\ref{taw}) leads to the following.

\begin{equation}
\begin{tabular}{|c|c|} \hline
Figure 1 & $|\De_1^d|$ \\ \hline
(b), (f), (g), (i), (j)
& $0$ \\
(h), (k)
& $n$
\\ \hline
\end{tabular}
\label{taz}
\end{equation}

Let us discuss Figures 1(b), 1(g) and 1(h)
as examples to illustrate (\ref{taz}).
We again index the vertices by Figure 5.
For vertex $i$, its root vector is denoted by $X_i$.

In Figure 1(b),
we have $\si X_m = X_{m+1}$ for some root vectors $X_m, X_{m+1}$. Then
\[ \si[X_m, X_{m+1}] = [\si X_m, \si X_{m+1}]
= [X_{m+1}, X_m] = - [X_m, X_{m+1}] ,\]
so $\al_m + \al_{m+1} \in \De_{-1}^d$.
Similarly, all the elements of $\De^d \cap \De^+$ in (\ref{taw}) belong to
$\De_{-1}^d$. Hence $|\De_1^d| = 0$ for Figure 1(b) in (\ref{taz}).

In Figure 1(g), let $\al_m$ be the circled vertex. Then
\begin{equation}
\si X_m = - X_m \;,\;
\begin{array}{l}
\si [[X_{m-1}, X_m], X_{m+1}] = [[\si X_{m-1}, \si X_m], \si X_{m+1}] \\
= [[X_{m+1}, - X_m], X_{m-1}] = - [[X_{m-1}, X_m], X_{m+1}] ,
\end{array}
\label{qoo}
\end{equation}
hence $\al_m, \al_{m-1} + \al_m + \al_{m+1} \in \De_{-1}^d$.
By similar computation, all the elements of $\De^d \cap \De^+$
in (\ref{taw}) belong to
$\De_{-1}^d$. Hence $|\De_1^d| = 0$ for Figure 1(g) in (\ref{taz}).

In Figure 1(h), let $\al_m$ be the uncircled vertex. Then (\ref{qoo}) becomes
\[
\si X_m = X_m \;,\;
\begin{array}{l}
\si [[X_{m-1}, X_m], X_{m+1}] = [[\si X_{m-1}, \si X_m], \si X_{m+1}] \\
= [[X_{m+1}, X_m], X_{m-1}] = [[X_{m-1}, X_m], X_{m+1}] .
\end{array}
\]
Hence $\al_m, \al_{m-1} + \al_m + \al_{m+1} \in \De_1^d$.
By similar computation, all the elements of $\De^d \cap \De^+$ in (\ref{taw}) belong to
$\De_1^d$. Doubling the number to include negative roots,
we have $|\De_1^d| = 2m = n$ for Figure 1(h) in (\ref{taz}).
The rest of (\ref{taz}) are computed similarly.

For references here and later, we provide the dimensions of
complex classical Lie algebras, and their numbers of roots $|\De|$.
See for example \cite[Appendix C]{kn}.
\begin{equation}
\begin{tabular}{|c|c|c|} \hline
$L$ & $\dim L$ & $|\De|$ \\ \hline
$A_n$ & $n^2+2n$ & $n^2 + n$ \\
$B_n$ & $2n^2+n$ & $2n^2$ \\
$C_n$ & $2n^2+n$ & $2n^2$ \\
$D_n$ & $2n^2-n$ & $2n^2-2n$
\\ \hline
\end{tabular}
\label{tas}
\end{equation}

In (\ref{taw}), the roots in $\De^d \cap \De^+$
are linearly independent.
By Proposition \ref{lmao}(c), (\ref{taz}) and (\ref{tas}),
\begin{equation}
\begin{tabular}{|c|c|} \hline
Figure 1 & $\dim L^\si$ \\ \hline
(b) & $\frac{1}{2}(n^2+n)$ \\
(f), (g), (k) & $n^2$ \\
(h) & $n^2 + n$ \\
(i), (j) & $n^2 - n$
\\ \hline
\end{tabular}
\label{pcc}
\end{equation}

For example in Figure 1(b),
$\dim L^\si = \frac{1}{2}|\De(A_n)| = \frac{1}{2}(n^2+n)$.
And in Figure 1(h),
$\dim L^\si = |\KK| + \frac{1}{2} |\De(C_n)| = n + n^2$.

Our motivation for computing $\dim L^\si$ is to find $\si$.
It is known that $L^\si$ determines $\si$
\cite[Ch.X-6,Thm.6.2]{he};
namely if $L^\si \cong L^\tau$,
then $\si$ and $\tau$ are conjugated by $\aut(L)$.
But here we merely have $\dim L^\si$ instead of $L^\si$,
and we will apply
\'{E}. Cartan's observation that
``$\dim L^\si$ determines
$\si$ in most cases''.

Let $\fg$ be a real form of $L$,
namely $\fg$ is a real subalgebra such that $\fg \otimes \bc = L$.
Let $\fg =\fk + \fp$ be a Cartan decomposition,
where $\fk$ is a maximal compact subalgebra.
Cartan defines the character of
$\fg$ by \cite[p.263]{ca}
\begin{equation}
\de = \dim \fp - \dim \fk .
\label{chara}
\end{equation}
He observes that $\de$ determines $\fg$
in all exceptional cases, as well as most classical cases.
In fact he uses it to denote the exceptional real forms,
for instance $\fe_{6(-26)}$ is the real form of $E_6$
with character $\de = -26$
\cite[Ch.X-6,Table V]{he}.
Given $\si \in \inv(L)$, it corresponds to
a real form $\fg = \fk + \fp$ by
\[ \si(\fg) = \fg \,,\, L^\si \cap \fg = \fk \,,\,
L^{-\si} \cap \fg = \fp , \]
namely $\si$ is a Cartan involution of $\fg$.
Then $\de + 2 \dim L^\si = \dim L$,
so we can discuss in terms of $\dim L^\si$ instead of $\de$.

\begin{theorem}
Let $L$ be a complex simple Lie algebra,
and $\si, \tau \in \inv(L)$.

\noindent {\rm (a)} $\;$ If $L = A_n$ and
$\dim L^\si = \frac{1}{2}(n^2 +n)$, then
$L^\si = \fs \fo(n+1, \bc)$.

\noindent {\rm (b)} $\;$ If $L = D_n$ and
$\dim L^\si = n^2 - n$, then $L^\si = \fs \fo(n, \bc)^2$.

\noindent {\rm (c)} $\;$ If $L = D_n$, $L^\si \not\cong L^\tau$
and $\dim L^\si = \dim L^\tau$, then $n>4$ is a perfect square and
$\{L^\si, L^\tau\} = \{A_{n-1} + \bc,
\fs \fo(n + \sqrt{n}, \bc) + \fs \fo(n - \sqrt{n}, \bc)\}$.

\noindent {\rm (d)} $\;$ If $L$ is not of type $A$ or $D$, and
$\dim L^\si = \dim L^\tau$, then $L^\si \cong L^\tau$.
\label{cahe}
\end{theorem}
\begin{proof}
If $L$ is not of type $A$ or $D$, then its real forms are
uniquely determined by their characters \cite[Ch.X-6,p.517]{he},
so $\si$ is determined by $\dim L^\si$. This proves part (d).

We say that $\fg$ is split or normal if it has a Cartan
subalgebra contained in $\fp$.
The character of a split form is strictly larger than the characters
of other real forms of $L$ \cite[Ch.X-6,p.517]{he}.
So if the split form corresponds to $\si$, then
$\dim L^\si < \dim L^\tau$ for all $\tau \in \inv(L)$
such that $L^\si \not\cong L^\tau$.

The split form of $A_n$ is $\fs \fl(n+1, \br)$,
and it corresponds to $L^\si = \fs \fo(n+1, \bc)$.
Since $\dim \fs \fo(n+1, \bc) = \frac{1}{2}(n^2 +n)$,
this proves part (a).

The split form of $D_n$ is $\fs \fo(n,n)$,
and it corresponds to $L^\si = \fs \fo(n, \bc)^2$.
Since $\dim \fs \fo(n,\bc)^2 = n^2 -n$, this proves part (b).

Finally we prove part (c). Let $L = D_n$.
The Kac diagrams in \cite[Figs.1-3]{c1} show that possible $L^\si$ are
\begin{equation}
\begin{array}{cl}
\mbox{(a)}
& L^\si = A_{n-1} + \bc \;,\;
\dim L^\si = n^2 ,\\
\mbox{(b)}
& L^\si = \fs \fo(p, \bc) + \fs \fo(2n-p, \bc) \;,\;
\dim L^\si
= p^2 + 2n^2 -2np - n .
\end{array}
\label{cox}
\end{equation}

If (\ref{cox})(a) and (\ref{cox})(b) have equal
$\dim L^\si$, then $n^2 = p^2 + 2n^2 -2np - n$.
Hence $(n-p)^2 = n$, namely
$n$ is a perfect square and $p = n \pm \sqrt{n}$.
This leads to the pair
$\{A_{n-1} + \bc,
\fs \fo(n + \sqrt{n}, \bc) + \fs \fo(n - \sqrt{n}, \bc)\}$
in part (c). For $n=4$,
$A_3 + \bc$ and
$\fs \fo(6, \bc) + \fs \fo(2, \bc)$
are isomorphic.

We also compare (\ref{cox})(b) with itself for various $p$ and $q$.
Suppose that $L^\si = \fs \fo(p, \bc) + \fs \fo(2n-p, \bc)$ and
$L^\tau = \fs \fo(q, \bc) + \fs \fo(2n-q, \bc)$ have equal dimension.
Then $p^2-2np=q^2-2nq$. This implies that $q \in \{2n-p,p\}$,
so $L^\si \cong L^\tau$.
We conclude that for $L = D_n$, the only case of $L^\si \not\cong L^\tau$
and $\dim L^\si = \dim L^\tau$ is listed in part (c).
This completes the proof of Theorem \ref{cahe}.\end{proof}


\newpage
\section{Invariant Subalgebras}
\setcounter{equation}{0}

Proposition \ref{licy} says that every affine Vogan diagram
in Figure 1 represents an $L$-involution $\si$.
In this section, we find $\si$
on a case-by-case basis. We show that
it is given by the Kac diagram,
or equivalently the invariant subalgebra $L^\si$ in Figure 1.
We then complete the proof of Theorem \ref{thm1}.

In Figure 1, we assume that $L$ has rank $n$,
so the affine Vogan diagrams have $n+1$ vertices.
Their vertices are indexed $0, 1, ..., n$ by Figure 5.

\sk

\noindent \underline{Figure 1(a)}

Here $n$ is odd. Let $m = \frac{n-1}{2}$.
Write an $(n+1) \times (n+1)$ matrix as
$\left(
\begin{array}{cc}
s & t \\
u & v
\end{array} \right)$,
where $s,t,u,v$ are $(m+1) \times (m+1)$ matrices.
Then the affine Vogan diagram in Figure 1(a) represents
the involution
$\si \left(
\begin{array}{cc}
s & t \\
u & v
\end{array} \right) =
\left(
\begin{array}{cc}
v & u \\
t & s
\end{array} \right)$, so
\[ L^\si = \left\{ \left(
\begin{array}{cc}
s & t \\
t & s
\end{array} \right)
\in \fs \fl (n+1 , \bc) \right\}
= A_m^2 + \bc .\]
This solves Figure 1(a).

\sk

For the rest of Figure 1,
we let the affine Vogan diagram $(c,d)$ on $\dyo$
represent $\si$, and
let the Kac diagram represent $\tau$.
We apply Theorem \ref{cahe} to find $L^\si$.
Recall that $\De^d$ and $\De_{\pm 1}^d$
are defined in (\ref{ink}) and (\ref{ksi}).

\sk

\noindent \underline{Figure 1(b)}

By (\ref{pcc}),
$\dim L^\si = \frac{1}{2}(n^2 + n)$.
By Theorem \ref{cahe}(a), $L^\si = \fs \fo(n+1, \bc)$.
This solves Figure 1(b).

\sk

\noindent \underline{Figure 1(c)}

Express the roots as $\pm e_i \pm e_j$ and $\pm e_i$
\cite[Appendix C]{kn}.
Let the vertices of $\dyo$ represent
$-e_1-e_2, e_1 - e_2, ..., e_{n-1} - e_n, e_n$.
So $d$ switches $e_1$ with $-e_1$,
and fixes other $e_i$. Therefore,
\[ \De \bsl \De^d =
\{\pm e_1 \pm e_i \;;\; i= 2, ..., n\}
\cup \{\pm e_1\}. \]
Hence $|\De \bsl \De^d| = 4(n-1) + 2 = 4n-2$.
We also have $\dim \langle \De^d \rangle = n-1$
because $\langle \De^d \rangle$ is spanned by
$e_2, ..., e_n$. So by Proposition \ref{lmao}(b),
\begin{equation}
\begin{array}{rl}
\dim L^\si & = \dim \langle \De^d \rangle +
|\De_1^d| + \frac{1}{2} |\De \bsl \De^d| \\
& = n-1 + |\De_1^d| + \frac{1}{2}(4n-2) \\
& = 3n-2 + |\De_1^d| .
\end{array}
\label{asa}
\end{equation}

Vertex $e_n$ is circled, so $e_n \in \De_{-1}^d$
and $\De_1^d = \{\pm e_i \pm e_j \;;\;
1 < i < j \leq n\}$. Therefore,
\begin{equation}
|\De_1^d| = 2(n-1)(n-2)= 2n^2 - 6n + 4 .
\label{ati}
\end{equation}

By (\ref{asa}) and (\ref{ati}),
\begin{equation}
\dim L^\si = (3n-2) + (2n^2 -6n +4)
= 2n^2 - 3n + 2 .
\label{aem}
\end{equation}

For the Kac diagram of Figure 1(c),
\begin{equation}
\dim L^\tau = \dim(B_{n-1} + \bc) = 2(n-1)^2 + (n-1) + 1
= 2n^2 - 3n + 2.
\label{ali}
\end{equation}
By Theorem \ref{cahe}(d), (\ref{aem}) and (\ref{ali}),
$L^\si = L^\tau = B_{n-1} + \bc$. This solves Figure 1(c).

\sk

\noindent \underline{Figures 1(d)}

In the affine Vogan diagram,
the root $e_m - e_{m+1}$ is circled. So
\[ \De_1^d = \{\pm e_i \pm e_j \;;\; 1 < i < j \leq m
\mbox{ or } m+1 \leq i < j \leq n\}
\cup \{ \pm e_{m+1} ,..., \pm e_n\} .\]
Hence
\begin{equation}
\begin{array}{rl}
|\De_1^d| & = 2(m-1)(m-2) + 2(n-m)(n-m-1) + 2(n-m) \\
& = 4m^2 - 6m +4 + 2n^2 -4nm.
\end{array}
\label{duu}
\end{equation}
By (\ref{asa}) and (\ref{duu}),
\begin{equation}
\dim L^\si = (3n-2) + (4m^2 -6m +4 +2n^2 -4nm).
\label{tii}
\end{equation}

The Kac diagram in Figure 1(d) satisfies
\[
\begin{array}{rl}
\dim L^\tau & = \dim D_{n-m+1} + \dim B_{m-1}\\
& = 2(n-m+1)^2 -(n-m+1) + 2(m-1)^2 + (m-1) ,
\end{array}
\]
which equals (\ref{tii}).
By Theorem \ref{cahe}(d),
$L^\si = L^\tau = D_{n-m+1} + B_{m-1}$.
This solves Figure 1(d).

\sk

\noindent \underline{Figures 1(e)}

The affine Vogan diagram satisfies
\[ \De_1^d = \{\pm e_i \pm e_j \;;\;
1 < i < j \leq n \} \cup
\{\pm e_2, ..., \pm e_n\} .\]
Hence
\begin{equation}
 |\De_1^d| = 2(n-1)(n-2) + 2(n-1)
= 2n^2 - 4n +2 .
\label{zoom}
\end{equation}
By (\ref{asa}) and (\ref{zoom}),
\begin{equation}
\dim L^\si = (3n-2) + (2n^2 - 4n +2)
= 2n^2 -n.
\label{tip}
\end{equation}
The Kac diagram satisfies
$\dim L^\tau = \dim D_n = 2n^2 -n$,
which equals (\ref{tip}).
By Theorem \ref{cahe}(d), $L^\si = L^\tau = D_n$.
This solves Figure 1(e).

\sk

\noindent \underline{Figures 1(f), 1(g)}

By (\ref{pcc}),
$\dim L^\si = n^2$.
The Kac diagram gives $L^\tau = A_{n-1} + \bc$,
so $\dim L^\tau = n^2$.
By Theorem \ref{cahe}(d), $L^\si = L^\tau = A_{n-1} + \bc$.
This solves Figures 1(f) and 1(g).

\sk

\noindent \underline{Figure 1(h)}

By (\ref{pcc}),
$\dim L^\si = n^2 + n$.
The Kac diagram
with black vertex $m = \frac{n}{2}$ gives $L^\tau = C_m^2$, so
$\dim L^\tau = 2 (2m^2 + m) = n^2 + n$.
 By Theorem \ref{cahe}(d), $L^\si = L^\tau = C_m^2$.
This solves Figure 1(h).

\sk

\noindent \underline{Figures 1(i), 1(j)}

By (\ref{pcc}), $\dim L^\si = n^2-n$.
By Theorem \ref{cahe}(b), $L^\si = \fs \fo(n, \bc)^2$.
This solves Figures 1(i) and 1(j).

\sk

\noindent \underline{Figure 1(k)}

Here $n$ is even.
By (\ref{pcc}) and Figure 1(k),
\begin{equation}
\dim D_n^\si = n^2 \,,\, D_n^\tau = A_{n-1} + \bc \,,\,
\dim D_n^\tau = n^2 .
\label{cow}
\end{equation}

If $n$ is not a perfect square larger than 4, then
by Theorem \ref{cahe}(c) and (\ref{cow}),
$D_n^\si = A_{n-1} + \bc$ and
we have solved Figure 1(k).

Suppose that
$n \in \{4^2, 6^2, 8^2, ...\}$.
By Theorem \ref{cahe}(c) and (\ref{cow}),
\begin{equation}
D_n^\si \in \{A_{n-1} + \bc ,
\fs \fo(n + \sqrt{n}, \bc) + \fs \fo(n - \sqrt{n}, \bc)\} .
\label{perf}
\end{equation}
Consider the similar affine Vogan diagram of rank $n+2$,
namely there are two more white vertices.
It represents an involution on $D_{n+2}$, still denoted by $\si$.
We have $D_{n+2}^\si = A_{n+1} + \bc$
because $n+2$ is not a perfect square.
The natural inclusion
$D_n \hookrightarrow D_{n+2}$ intertwines with $\si$, hence
\begin{equation}
 D_n^\si \hookrightarrow D_{n+2}^\si = A_{n+1} + \bc .
 \label{lic}
 \end{equation}
By (\ref{perf}) and (\ref{lic}), we have
$D_n^\si = A_{n-1}+ \bc$ because
$\fs \fo(n + \sqrt{n}, \bc) + \fs \fo(n - \sqrt{n}, \bc)$
cannot be imbedded into $A_{n+1}$.
This solves Figure 1(k) for all $n$.

\sk

\noindent \underline{Figure 1(l)}

Express the roots as $\pm e_i \pm e_j$ \cite[Appendix C]{kn}.
Let the vertices of $\dyo$ represent
$-e_1 -e_2, e_1 - e_2, ..., e_{n-1}- e_n, e_{n-1} +e_n$.
So $d$ switches $e_1$ with $-e_1$,
switches $e_n$ with $-e_n$, and fixes other $e_i$. Hence
\[ \De^d = \{\pm e_i \pm e_j \;;\; 1 < i < j < n\} ,\]
and so $|\De^d| = 2(n-2)(n-3)$. Therefore,
\begin{equation}
|\De \bsl \De^d| = (2n^2 - 2n) - 2(n-2)(n-3) = 8n-12.
\label{bdu}
\end{equation}

A simple system of $\De^d$ is
$e_2-e_3, ..., e_{n-2}-e_{n-1},e_{n-2}+e_{n-1}$, so
$\dim \langle \De^d \rangle = n-2$.
Together with Proposition \ref{lmao}(b) and (\ref{bdu}),
\begin{equation}
\begin{array}{rl}
\dim L^\si
& = \dim \langle \De^d \rangle + |\De_1^d| +
\frac{1}{2} |\De \bsl \De^d| \\
& = n-2 + |\De_1^d| + 4n-6  = 5n-8 + |\De_1^d| .
\end{array}
\label{bsa}
\end{equation}

We have $|\De_1^d| =  |\De^d| = 2(n-2)(n-3)$,
so by (\ref{bsa}),
\[ \dim L^\si = 5n-8  + 2(n-2)(n-3)
= 2n^2 - 5n + 4.\]
The Kac diagram satisfies $L^\tau = D_{n-1} + \bc$, so
\[ \dim L^\tau = \dim (D_{n-1} + \bc) = 2(n-1)^2 - (n-1) +1
= 2n^2 - 5n + 4 .\]
Hence $\dim L^\si = \dim L^\tau$.

If $n=4$, then $L^\si = L^\tau$ by Theorem \ref{cahe}(c).
If $n > 4$, then $L^\tau$ is neither $A_{n-1} + \bc$ nor
$\fs \fo(n + \sqrt{n}, \bc) + \fs \fo(n - \sqrt{n}, \bc)$,
so again $L^\si = L^\tau$ by Theorem \ref{cahe}(c).
This solves Figure 1(l).

\sk

\noindent \underline{Figure 1(m)}

In the affine Vogan diagram, the root $e_m - e_{m+1}$ is circled.
So we have
\[ \De_1^d = \{\pm e_i \pm e_j \;;\;
1 < i < j \leq m \mbox{ or } m+1 \leq i < j < n\} ,\]
and hence
\begin{equation}
 |\De_1^d| = 2(m-1)(m-2) + 2(n-m-1)(n-m-2) .
 \label{ttn}
 \end{equation}
By (\ref{bsa}) and (\ref{ttn}),
\begin{equation}
\begin{array}{rl}
\dim L^\si
& = 5n-8 + 2(m-1)(m-2) + 2(n-m-1)(n-m-2) \\
& = -n + 4m^2 + 2n^2 -4nm .
\end{array}
\label{li}
\end{equation}

The Kac diagram satisfies $L^\tau = D_m + D_{n-m}$, so
\[ \dim L^\tau = \dim D_m + \dim D_{n-m}
= 2m^2 -m + 2(n-m)^2 - (n-m) ,\]
which equals (\ref{li}).
We conclude that
\begin{equation}
\dim L^\si = \dim L^\tau \;,\;
L^\tau = D_m + D_{n-m} .
\label{jol}
\end{equation}

If $n=4$, then Theorem \ref{cahe}(c) and (\ref{jol}) imply
that $L^\si \cong L^\tau$. Suppose that $n > 4$.
We remove the circled vertex from the affine Vogan diagram.
It results in the Vogan diagram of $D_m +D_{n-m}$,
and the diagram involution leads to
\[ B_{m-1} + B_{n-m-1} = D_m^\si + D_{n-m}^\si
\subset L^\si .\]
For $n > 4$,
$B_{m-1} + B_{n-m-1}$ cannot be imbedded into $A_{n-1}$, so
\begin{equation}
L^\si \neq A_{n-1} + \bc .
\label{pon}
\end{equation}
By Theorem \ref{cahe}(c), (\ref{jol}) and (\ref{pon}),
it follows that $L^\si = D_m + D_{n-m}$.

\sk

The following proposition will help to handle
Figures 1(n) and 1(o).

\begin{proposition}
$\,$

\noindent {\rm (a)} $\,$ $A_7$ and $D_6$ do not contain
any subalgebra isomorphic to $F_4$;

\noindent {\rm (b)} $\,$ $D_6$ does not contain any subalgebra
isomorphic to $C_4$.
\label{bun}
\end{proposition}
\begin{proof}
Let $L$ denote $A_7$ or $D_6$, and suppose that
$F_4$ is imbedded into $L$.
Consider the natural representation
$\pi: L \lra \mbox{End} (\bc^n)$,
where $n=8$ for $L= \fs \fl(8,\bc)$,
and $n=12$ for $L= \fs \fo(12,\bc)$.
It leads to an $F_4$-representation of dimension $\leq 12$.
But the Weyl dimension formula \cite[V-6]{kn} shows that
the only $F_4$-representation of dimension $\leq 12$ is the
trivial representation,
so $F_4$ is imbedded into the kernel of $\pi$.
Hence the kernel of $\pi$ is a nontrivial ideal of $L$,
which is a contradiction.
This proves part (a).
Part (b) follows from the lists of subalgebras of $D_6$ in
\cite[Tables IV,XIII]{lg}.
\end{proof}

\sk

\noindent \underline{Figure 1(n)}

Remove two endpoints from the affine Vogan diagram
to obtain the Vogan diagram of $E_6$.
Its diagram involution leads to
$F_4 = E_6^\si \subset E_7^\si$.
By \cite[Fig.1-2]{c1},
\begin{equation}
E_7^\si \in \{E_6 + \bc \;,\; A_7 \;,\; D_6 + A_1\} .
\label{jas}
\end{equation}
But Proposition \ref{bun}(a) has excluded
$A_7$ and $D_6 + A_1$,
hence $E_7^\si = E_6 + \bc$.
This solves Figure 1(n).

\sk

Two affine Vogan diagrams are said to be equivalent
if they represent the same involution.
Given an affine Vogan diagram with circled vertex $\al$,
we define
\[
\begin{array}{cl}
F_\al: &
\mbox{reverse the circlings of all vertices $\be$ adjacent to $\al$,}\\
& \mbox{except when $\be$ is a longer root joint to $\al$
by a double edge,}\\
& \mbox{or when $\be$ is not fixed by the diagram involution.}
\end{array}
\]
Two affine Vogan diagrams are equivalent if and only if they are related
by a sequence of such $F_\al$. Here Figure 1(o) and Figure 1(p)
are equivalent to each other, and they are not equivalent to Figure 1(n)
\cite[Table 2]{ch2}\cite[Thm.5.1]{ch}.

\sk

\noindent \underline{Figure 1(o), 1(p)}

In the affine Vogan diagram of Figure 1(o),
remove the circled vertex
to obtain the Vogan diagram of $A_7$.
Its diagram involution leads to
$C_4 = A_7^\si \subset E_7^\si$.

The affine Vogan diagrams in Figures 1(n) and 1(o)
are not equivalent.
Since $E_6 + \bc$ is the invariant subalgebra of Figure 1(n),
by (\ref{jas}),
we are left with $A_7$ and $D_6 + A_1$
as candidates for $E_7^\si$ in Figure 1(o).
By Proposition \ref{bun}(b), it cannot be $D_6 + A_1$.
Hence $E_7^\si = A_7$ for Figure 1(o).
The affine Vogan diagrams in Figures 1(o) and 1(p) are equivalent.

\sk

\noindent {\it Proof of Theorem \ref{thm1}:}

Theorem \ref{thm1}(a) follows from Proposition \ref{exx}.
The first part of Theorem \ref{thm1}(b) follows from
Proposition \ref{licy}.
It remains to show that the involutions are given by
the accompanying Kac diagrams and invariant subalgebras
in Figure 1,
and this has been checked case-by-case
in this section. \qed


\newpage
\section{Associate Symmetric Pairs} \label{as}
\setcounter{equation}{0}

In this section, we prove Theorem \ref{thmcorr}.
Let $L$ be a complex simple Lie algebra
with real form $\fg$.
The Cartan involution $\thq$ of $\fg$ extends to $L$
by complex linearity.
There exists a Cartan subalgebra $H$ of $L$
such that $\thq$ stabilizes $H$ and
a simple system $\Pi$.
We have the sum of eigenspaces $H = H^\thq + H^{-\thq}$.
Let $\De \subset H^*$ be the roots, and
let $\De^\thq$ be their restrictions to $H^\thq$.
We have the root space decomposition
\begin{equation}
 L = H^\thq + \sum_{\De^\thq} L_x .
 \label{0307}
 \end{equation}

The Kac diagram of $\fg$ appears on $\dyr$,
whose vertices represent a simple system of $\De^\thq$
with a lowest weight.
We have
\begin{equation}
\begin{array}{cl}
\mbox{(a)} & H = H^\thq \; \Longleftrightarrow \;
\thq \mbox{ acts trivially on } \Pi
\; \Longleftrightarrow \;
r =1 ,\\
\mbox{(b)} & H^{-\thq} \neq 0 \; \Longleftrightarrow \;
\thq \mbox{ permutes } \Pi \mbox{ nontrivially }
\; \Longleftrightarrow \;
r =2 .
\end{array}
\label{bes}
\end{equation}

Let $\pi : \De \lra \De^\thq$ be the restriction of
$\De$ to $H^\thq$.
In (\ref{bes})(b), $\pi$ is 1-to-1 or 2-to-1.
Hence for all $x \in \De^\thq$, $\dim L_x$ is 1 or 2.
If $\dim L_x =2$, there exist distinct $\al, \be \in \De$ such that
$\pi(\al) = \pi(\be) = x$, and
\begin{equation}
 L_\al + L_\be = L_x = (L_x \cap L^\thq) + (L_x \cap L^{-\thq}) .
 \label{0308}
 \end{equation}

As explained in (\ref{idsp}),
the symmetric pairs $(\fg, \fg^\si)$ are identified with
commuting $L$-involutions $(\thq, \si)$.
The next definition provides their diagrammatic objects.
Recall that the Kac diagrams and affine Vogan diagrams
on $\dyr$ are defined in Definitions \ref{dkd} and \ref{dvd}.

\begin{definition}
{\rm \cite[Def.1.2]{ch}
A {\it double Vogan diagram} on $\dyr$ is a triple $(p,c,d)$,
where $p$ is a Kac diagram, $(c,d)$ is an affine Vogan diagram,
and $d$ preserves $p$.
We say that $(p,c,d)$ represents commuting $(\thq, \si)$ if
$p$ represents $\thq$ in the sense of Theorem \ref{kd}, and
\begin{equation}
\begin{array}{l}
\si (L_x \cap L^{\pm \thq})
= L_{dx} \cap L^{\pm \thq} \mbox{ for all } x \in \dyr, \\
\mbox{and whenever } dx = x:\\
\mbox{$x$ white, uncircled} \; \Longrightarrow \;
\mbox{$\si =1$ on $L_x \cap L^\thq$}, \\
\mbox{$x$ white, circled} \; \Longrightarrow \;
\mbox{$\si = -1$ on $L_x \cap L^\thq$}, \\
\mbox{$x$ black, uncircled} \; \Longrightarrow \;
\mbox{$\si =1$ on $L_x \cap L^{-\thq}$}, \\
\mbox{$x$ black, circled} \; \Longrightarrow \;
\mbox{$\si = -1$ on $L_x \cap L^{-\thq}$}.
\end{array}
\label{ten}
\end{equation}
\label{suim}
}
\end{definition}

The condition $d$ preserves $p$ means that
all $\al$ and $d \al$ have the same color.
It ensures that $\thq$ and $\si$ commute.

For all $x \in \dyo$ and some $x \in \dyt$,
we have $\dim L_x = 1$. In that case
(\ref{ten}) simplifies to (\ref{dew}),
namely the circling of $x$ determines $\si|_{L_x} = \pm 1$.
But for $\dim L_x = 2$ in $\dyt$,
we need the more complicated (\ref{ten})
because $\si$
may have distinct eigenvalues on $L_x$.
We will further discuss this issue in
the next section.
Nevertheless for the sake of proving Theorem \ref{thmcorr},
all vertices $x$ in Figure 2 other than
the endpoints of $\dyt$ in Figure 2(c)
satisfy $\dim L_x =1$, so we can repeat the
earlier arguments.

The explicit matchings
between double Vogan diagrams and symmetric pairs are given in
\cite[\S7]{ch}.
However, one may mismatch the double Vogan diagrams
of associate pairs
$(\fg, \fg^\si)$ and $(\fg, \fg^{\thq \si})$,
especially when they are
not extensions of Vogan diagrams on $\dy$.
For the three associate pairs in (\ref{rfi}),
such errors occur in
pages 1741, 1742, 1747 of \cite{ch}.
Figure 2 has interchanged
$(\fg, \fg^\si)$ and $(\fg, \fg^{\thq \si})$.
We now apply Theorem \ref{thm1} to show that
Figure 2 provides the correct matchings.

\sk

\noindent {\it Proof of Theorem \ref{thmcorr}:}

Two double Vogan diagrams represent associate pairs
if they coincide on the white vertices, and their circlings on
the black vertices are opposite. We first consider Figure 2(a),
\[ (\fs \fp(m,m),\fs \fp(m,\bc)) \,,\,
(\fs \fp(m,m), \fg \fl(m,\bh)) .\]
The left affine Vogan diagram is
Figure 1(h), namely $L^\si = C_m^2$.
The right affine Vogan diagram is
Figure 1(g), namely $L^\si = A_{n-1}+ \bc$, where $n=2m$. We have
\[ \fs \fp(m,\bc)^2 \otimes \bc = C_m^2 \; \mbox{ and } \;
\fg \fl(m, \bh) \otimes \bc = A_{n-1}+ \bc , \]
so the symmetric pairs and diagrams are matched as
in Figure 2(a).

Next we consider Figure 2(b),
\[ (\fs \fo(n,n), \fg \fl (n,\br)) \,,\,
(\fs \fo(n,n), \fs \fo(n,\bc)) \,,\, n \mbox{ even}.\]
The left affine Vogan diagram is Figure 1(k),
namely $L^\si = A_{n-1}+\bc$.
The right affine Vogan diagram is Figure 1(j),
namely $L^\si = D_m^2$, where $n=2m$. We have
\[ \fg \fl(n, \br) \otimes \bc = A_{n-1}+\bc \; \mbox{ and } \;
 \fs \fo(n, \bc) \otimes \bc = D_m^2 ,\]
so the symmetric pairs and diagrams are matched as
in Figure 2(b).

Finally we consider Figure 2(c),
\[ (\fs \fo(n,n), \fg \fl (n,\br)) \,,\,
(\fs \fo(n,n), \fs \fo(n,\bc)) \,,\, n \mbox{ odd}.\]
Here we cannot apply Figure 1
because the affine Vogan diagrams are drawn on $\dyt$.
Nevertheless each $x \in \dyt$ has $\dim L_x =1$
except for the endpoints.
So we can compare the dimensions of their
invariant subalgebras by the method used in Section 3.

Let the left and right double Vogan diagrams of Figure 2(c)
represent commuting $L$-involutions
$(\thq, \si_1)$ and $(\thq,\si_2)$ respectively.
We compare the dimensions of $L^{\si_1}$
and $L^{\si_2}$. Since the middle vertex of the left diagram
(resp. right diagram) is not circled (resp. is circled),
$\si_1$ acts as identity on more root spaces than $\si_2$, hence
\begin{equation}
\dim L^{\si_1} > \dim L^{\si_2} .
\label{csa}
\end{equation}
We also have real dimensions
\begin{equation}
\dim \fg \fl(n, \br) = n^2 \;,\;
\dim \fs \fo(n, \bc) = n^2 - n .
\label{cdu}
\end{equation}
By (\ref{csa}) and (\ref{cdu}), it follows that
$L^{\si_1} = \fg \fl(n, \br) \otimes \bc$ and
$L^{\si_2} = \fs \fo(n, \bc) \otimes \bc$.
So the symmetric pairs and diagrams are matched as
in Figure 2(c).
This completes the proof of Theorem \ref{thmcorr}. \qed


\newpage
\section{Symplectic Symmetric Pairs}
\label{ss}
\setcounter{equation}{0}

In this section, we prove Theorem \ref{thm2} and Corollary \ref{fma}.
Let $\fg$ be a real form of $L$ with Cartan involution $\thq$,
and $\fk = \fg^\thq$ is a maximal compact subalgebra
of $\fg$.
We say that $\fg$ is of Hermitian type if at the Lie group level,
$G/K$ has a $G$-invariant Hermitian structure. In that case we also say
that $\thq$ is of Hermitian type.
We say that $\fg$ is of equal rank type if $\fg$ and $\fk$ have
the same rank. We say that $\thq$ is inner if it lies in the
identity component of the Lie group $\aut(\fg)$,
and we say that it is outer otherwise.
Let $\fz(\cdot)$ denote the center.
The Kac diagrams on $\dyr$ and the Vogan diagrams $(c,d)$ on $\dy$
can occur in three cases as listed in (\ref{big}) below.
Its arguments
can be found in Theorem \ref{kd} and
\cite[(2.2),(4,4)]{c1}.
Recall that $\sum_{\dyr} m_\al \al =0$.

\begin{equation}
\begin{tabular}{|c|c|c|c|c|} \hline
Kac diagram & Vogan diagram & $\thq$ & $\fg$ & $\fz(\fk)$ \\ \hline
$\dyo$, $\al, \be$ black, & $d=1$ & inner, & equal rank, & 1-dim \\
$m_\al = m_\be = 1$ & & Hermitian & Hermitian & \\
\hline
$\dyo$, $\al$ black, & $d=1$ & inner, & equal rank, & 0 \\
$m_\al = 2$ & & non-Hermitian & non-Hermitian & \\
\hline
$\dyt$, $\al$ black, & $d \neq 1$ & outer, & non-equal rank, & 0 \\
$m_\al = 1$ & & non-Hermitian & non-Hermitian &
\\ \hline
\end{tabular}
\label{big}
\end{equation}

Let us discuss the case where $\thq$ is outer.
It stabilizes a Cartan subalgebra $H$ of $L$,
and permutes a simple system nontrivially.
By (\ref{0307}), we have $L = H^\thq + \sum_{\De^\thq} L_x$,
where $\De^\thq$ are the restrictions of $\De$ to $H^\thq$.
Let $\pi: \De \lra \De^\thq$ be the restriction map.
By (\ref{0308}), if $\dim L_x =2$, then $L_x$
is the sum of $L_x \cap L^{\pm \thq}$,
as well as the sum $L_\al + L_\be$ where $\pi(\al)=\pi(\be)=x$.

\begin{lemma}
Let $\pi(\al)=\pi(\be)=x$, and
suppose that $\si \in \inv(L)$ stabilizes $L_x \cap L^{\pm \thq}$.
Then $\si$ has the same eigenvalue on $L_x \cap L^{\pm \thq}$
if and only if $\si$ stabilizes $L_\al$ and $L_\be$.
\label{reh}
\end{lemma}
\begin{proof}
Let $X_\al \in L_\al$ and $X_\be \in L_\be$ such that
$\thq X_\al = X_\be$ and
\[ L_x \cap L^{\pm \thq} = \bc(X_\al \pm X_\be) .\]

Suppose that $\si \in \inv(L)$ stabilizes $L_x \cap L^{\pm \thq}$.
Let $a,b \in \{\pm 1\}$ be their eigenvalues, namely
\[ \si (X_\al + X_\be) = a(X_\al + X_\be) \;,\;
\si (X_\al - X_\be) = b(X_\al - X_\be) .\]
It follows that
\[ \si X_\al =
\left\{
\begin{array}{cl}
\pm X_\al & \mbox{if } a=b ,\\
\pm X_\be & \mbox{if } a =-b .
\end{array}
\right. \]
Hence $\si$ stabilizes $L_\al$ and $L_\be$ if and only if
$a=b$, and the lemma follows.
\end{proof}

Let $c$ be a circling on $\dyt$.
Similar to Definition \ref{dvd}, we hope for
$\si \in \inv(L)$ such that for all $x \in \dyt$,
\begin{equation}
 \si|_{L_x} = \left\{
\begin{array}{cl}
1 & \mbox{if } x \mbox{ is uncircled} ,\\
-1 & \mbox{if } x \mbox{ is circled}.
\end{array}
\right.
\label{tem}
\end{equation}
The obstruction to (\ref{tem}) is that if $\dim L_x =2$,
then $\si$ may have distinct eigenvalues on the
subspaces of $L_x$.
This is handled by the next proposition.
Let $c(\dyt) \subset \dyt$ denote the circled vertices.

\begin{proposition}
Let $c$ be a circling on $\dyt$. There exists $\si \in \inv(L)$
which satisfies (\ref{tem}) if and only if $\sum_{c(\dyt)} m_x$ is even.
\label{jass}
\end{proposition}
\begin{proof}
By (\ref{big}),
fix a Kac diagram $p$ on $\dyt$ so that
it represents an outer $L$-involution $\thq$.
Let $c$ be a circling on $\dyt$.
Then $(p,c)$ is a double Vogan diagram on $\dyt$ with $d=1$,
and there exists $\si \in \inv(L)$ which commutes with $\thq$
and satisfies (\ref{ten})
\cite[Thm.1.3]{ch}\cite[Thm.1.3]{c1}.
So $\si$ stabilizes $L_x \cap L^{\pm \thq}$
for all $x \in \dyt$.
We aim to strengthen (\ref{ten}) to (\ref{tem}),
and the obstruction is whether $\si$ has the
same eigenvalue on $L_x \cap L^{\pm \thq}$.

By \cite[Prop.8.2(b)]{c1},
\begin{equation}
 \sum_{c(\dyt)} m_\al \mbox{ is even}
\; \Longleftrightarrow \; \si \mbox{ is inner}.
\label{tun}
\end{equation}

Recall from (\ref{cana}) that $\dyt = \pi(\dy) \cup \{\vf\}$.
Here $\si$ stabilizes the simple system represented by $\dy$.
If the equivalent conditions in (\ref{tun}) occur,
then $\si|_H =1$ and
$\si$ stabilizes $L_\al$ for all $\al \in \De$.
By Lemma \ref{reh}, $L_x$ is an eigenspace of $\si$.
So (\ref{ten}) can be strengthened to (\ref{tem}).

If the equivalent conditions in (\ref{tun}) fail,
then there exist distinct $\al, \be \in \dy$ such that
$\si$ interchanges $L_\al$ and $L_\be$.
Let $\pi(\al)= \pi(\be) = x \in \dyt$.
By Lemma \ref{reh}, $\si$ has distinct eigenvalues
on $L_x \cap L^{\pm \thq}$, so (\ref{tem}) is not satisfied.
This proves the proposition.
\end{proof}

By Proposition \ref{jass}, if $c$ is a circling on $\dyt$
such that $\sum_{c(\dyt)} m_x$ is even, then
$c$ determines $\si$ by (\ref{tem}).
If $\sum_{c(\dyt)} m_x$ is odd, then $c$ together with a
Kac diagram determine $(\thq, \si)$
by (\ref{ten}), but $c$ alone does not determine $\si$.

Recall that in (\ref{idsp}),
we identify commuting involutions $(\thq,\si)$
with symmetric pairs $(\fg,\fg^\si)$.
We say that $(\fg,\fg^\si)$ is
pseudo-Hermitian or symplectic if at the group level,
$G/G^\si$ carries such a $G$-invariant structure.
These notions can be replaced by
the following equivalent definitions,
which are more useful for us.
If $\fz(\fg^\si) \neq 0$, then
it is 1-dimensional, so it is
either compact ($\fz(\fg^\si) \subset \fk$) or split
($\fz(\fg^\si) \subset \fp$).

\begin{definition}
{\rm \cite[Thm.2.1]{bi}\cite[p.532]{sh}
The symmetric pair $(\fg, \fg^\si)$ is said to be
{\it symplectic} if $\fz(\fg^\si) \neq 0$.
In addition it is said to be {\it pseudo-Hermitian} if
$\fz(\fg^\si)$ is compact, and {\it non-pseudo-Hermitian}
if $\fz(\fg^\si)$ is split.}
\label{lad}
\end{definition}

Let us explain its motivation as discussed in \cite[\S3]{bi}.
Extend $\si$ to an $L$-involution by complex linearity.
Then $\si$ can be represented by a Vogan diagram
on $\dy$ with at most one black
vertex $\al$ \cite[Thm.6.96]{kn}.
Apply (\ref{big}) to $\si$ and we see that
\begin{equation}
\fz(L^\si) \neq 0 \; \Longleftrightarrow \; m_\al =1 \mbox{ and }
d=1 .
\label{pate}
\end{equation}
The condition $m_\al =1$ is expressed as $\mu^\De(h_\al^\De) =1$ in
\cite[Def.3.2]{bi}, where $\mu^\De$ is the highest root.

Let $G$ be a connected Lie group whose Lie algebra is $\fg$.
Let $d : \we^k \fg^* \lra \we^{k+1} \fg^*$ be the deRham operator
on the left invariant differential forms on $G$.
For $v \in \fg$, we define the contraction map and coadjoint map by
\[
\begin{array}{l}
\iota_v : \we^k \fg^* \lra \we^{k-1} \fg^* \;,\;
(\iota_v \om)(x_1, ..., x_{k-1}) =
\om(v,x_1, ..., x_{k-1}) ,\\
\ad_v^* : \we^k \fg^* \lra \we^k \fg^* \;,\;
(\ad_v^* \om)(x_1, ..., x_k) =
\sum_1^k \om(x_1, ..., [v,x_i], ..., x_k) .
\end{array}
\]
Suppose that the equivalent conditions in (\ref{pate}) hold.
Pick $0 \neq \phi \in \fz(\fg^\si)^*$
(denoted by $\be(h_\ga^\De, \cdot)$ in \cite[Thm.3.1(v)]{bi}),
and extend it to $\fg^*$ by the Killing form.
Then $d \phi$ satisfies
\begin{equation}
\begin{array}{cl}
\mbox{(a)} & \iota_v (d \phi) = \ad_v^* (d \phi) = 0
\mbox{ for all } v \in \fg^\si ,\\
\mbox{(b)} & \iota_v(d \phi) \neq 0
\mbox{ for all } v \in \fg \backslash \fg^\si .
\end{array}
\label{imp}
\end{equation}
If $\si$ is the differential of a $G$-involution
(still denoted by $\si$), then (\ref{imp})(a)
says that $d \phi$ is a $G$-invariant 2-form
on $G/G^\si$, and (\ref{imp})(b) says that
$d \phi$ is nondegenerate,
i.e. symplectic.

\sk

\noindent {\it Proof of Theorem \ref{thm2}:}

By Definition \ref{lad}, the symplectic non-pseudo-Hermitian
symmetric pairs $(\fg, \fg^\si)$
correspond to commuting $L$-involutions $(\thq, \si)$ such that
\begin{equation}
\begin{array}{cl}
\mbox{(a)} & \dim \fz(L^\si) =1 , \\
\mbox{(b)} & \thq = -1 \mbox{ on } \fz(L^\si) .
\end{array}
\label{oss}
\end{equation}
We look for double Vogan diagrams $(p,c,d)$ on $\dyr$
which represent (\ref{oss}).

We first prove Theorem \ref{thm2}(a).
Let $r=1$, namely $\fg$ is of equal rank type.
We claim that each of the following cases cannot occur,
\begin{equation}
\begin{array}{cl}
\mbox{(a)} & d \neq 1 \mbox{ and } d \mbox{ stabilizes } \dy ,\\
\mbox{(b)} & d = 1 .
\end{array}
\label{cto}
\end{equation}

By (\ref{big}),
condition (\ref{oss})(a) amounts to
$(c,d)$ being equivalent to a Kac diagram with
two black vertices,
namely $\si$ is of Hermitian type.
This excludes (\ref{cto})(a), because in that case
$\si$ is outer.
Suppose that (\ref{oss})(a) holds.
Then $\thq = 1$ on $\fz(L^\si)$ if and only if
$d=1$ \cite[Thm.1.1(a)]{mrl}.
Hence (\ref{oss})(b) excludes (\ref{cto})(b).
We have excluded both cases of (\ref{cto}),
hence $(c,d)$ is not an extension of
a Vogan diagram on $\dy$.
This proves Theorem \ref{thm2}(a).

Next we prove Theorem \ref{thm2}(b).
Let $r=2$, namely $\fg$ is of non-equal rank type.
Recall the diagram mapping $\pi: \dy \lra \dyt$ in (\ref{cana}).
Given a circling $c$ on $\dyt$, let $\pi^{-1}(c)$ be the circling
on $\dy$ such that $\al \in \dy$ is circled
if and only if $\pi(\al)$ is circled.

Assume for now that $d=1$.
We claim that (\ref{oss}) is equivalent to
\begin{equation}
\begin{array}{l}
\mbox{$\sum_{c(\dyt)} m_x$ is even,
$\pi^{-1}(c) =\{\al, \be\} \subset \dy$}\\
\mbox{with } m_\al = m_\be = 1 \mbox{ and } \pi(\al) = \pi(\be) .
\end{array}
\label{sso}
\end{equation}
To begin with, we need $\sum_{c(\dyt)} m_x$ is even
because otherwise $\si$ is outer \cite[Prop.8.2(b)]{c1},
which implies that $\fz(L^\si)=0$.
If $\sum_{c(\dyt)} m_x$ is even,
then Proposition \ref{jass} says that $c$ and $\pi^{-1}(c)$ represent
the same involution because
(\ref{ten}) and (\ref{tem}) are the same.
By (\ref{big}), the condition (\ref{oss})(a) is equivalent to
$\pi^{-1}(c) =\{\al, \be\} \subset \dy$ and $m_\al = m_\be = 1$,
namely $\si$ is a Hermitian involution.
Suppose that (\ref{oss})(a) holds.
Then (\ref{oss})(b) is equivalent to $\thq$ interchanging
$L_\al$ and $L_\be$ \cite[Prop.2.4(b)]{ch}
(the roles of $\thq$ and $\si$ are reversed in the formulation
of \cite[Prop.2.4(b)]{ch}),
which is is equivalent to $\pi(\al) = \pi(\be)$.
Hence when $d=1$, (\ref{oss}) is equivalent to (\ref{sso}).

Next we consider $d \neq 1$ on $\dyt$.
By \cite[Fig.3]{c1}, this happens only on
the Kac diagrams of
\begin{equation}
\begin{array}{cl}
\mbox{(a)} & \fg = \fs \fl(n,\br) \;,\; n \mbox{ even},\\
\mbox{(b)} & \fg = \fs \fo(n,n) \;,\; n \mbox{ odd}.
\end{array}
\label{ys}
\end{equation}

For $\fs \fl(n,\bc)$, an inner involution
is equivalent to a Hermitian involution or identity.
In (\ref{ys})(a),
if we express the simple system of $\fs \fl(n,\bc)$
as $e_1-e_2, ..., e_{n-1}-e_n$, then $d$
permutes $e_1$ and $e_n$ \cite[Fig.10]{c1}.
Hence $(c,d)$ represents an inner involution
if and only if $c$ alone represents an
inner involution, namely $c$ is trivial or
satisfies (\ref{sso}).
The resulting symmetric pairs are
$(\fs \fl (n, \br), \fs \fl(p,\br) + \fs \fl(q, \br) + \br)$,
with $p,q$ odd.

In (\ref{ys})(b), we see from Figure 2(c) that
\begin{equation}
\begin{array}{cl}
\mbox{(a)} &
c \mbox{ trivial } \, \Longleftrightarrow \,
\fg^\si = \fg \fl(n,\br)
\, \Longrightarrow \, \dim \fz(L^\si) = 1 ,\\
\mbox{(b)} &
c \mbox{ circles middle vertex } \, \Longleftrightarrow \,
\fg^\si = \fs \fo(n,\bc)
\, \Longrightarrow \, \dim \fz(L^\si) = 0 .
\end{array}
\label{sot}
\end{equation}
Hence (\ref{sot})(a) satisfies the conditions in
Theorem \ref{thm2}(b), and
$(\fs \fo(n,n),\fg \fl(n,\br))$ is indeed
symplectic non-pseudo-Hermitian
because $\fg \fl(n,\br)$ has a split center.
We conclude that when $d \neq 1$, (\ref{ys}) leads to
$(\fs \fl (n, \br), \fs \fl(p,\br) + \fs \fl(q, \br) + \br)$
with $p,q$ odd,
and $(\fs \fo(n,n),\fg \fl(n,\br))$.  \qed

\sk

\noindent {\it Proof of Corollary \ref{fma}:}

Suppose that $\fg$ is of equal rank type.
By Theorem \ref{thm2}(a), we find all the affine Vogan diagrams
on $\dyo$ which are equivalent to Kac diagrams with two black vertices,
and are not extensions of Vogan diagrams on $\dy$.
By Theorem \ref{thm1}(b), they are given by the first column of
\begin{equation}
 \mbox{Figures 1(a), 1(c), 1(f),
 1(g), 1(k), 1(l), 1(n).}
\label{wau}
\end{equation}

We extend the affine Vogan diagrams in (\ref{wau})
to double Vogan diagrams on $\dyo$, and
find their symmetric pairs in \cite[\S7]{ch}.
For example, the first item of (\ref{wau}) is Figure 1(a),
so we superimpose it with a Kac diagram
which is preserved by $d$.
This leads to the first diagram of Figure 3,
and its symmetric pair is
$(\fs \fu(n,n), \fs \fl(n,\bc)+ \br)$ \cite[p.1740]{ch}.
The next item of (\ref{wau}) is Figure 1(c).
We again superimpose it with a Kac diagram
to obtain another diagram in Figure 3,
and its symmetric pair is
$(\fs \fo(p,q), \fs \fo(p-1,q-1) + \fs \fo(1,1))$
for $p+q$ odd \cite[p.1740]{ch}.
We repeat this method with the rest of (\ref{wau})
and obtain Figure 3.

Next we consider $\fg$ of non-equal rank type.
As discussed in the proof of Theorem \ref{thm2}(b),
the case of $d \neq 1$ leads to
$(\fs \fl (n, \br), \fs \fl(p,\br) + \fs \fl(q, \br) + \br)$
with $p,q$ odd,
and also $(\fs \fo(n,n),\fg \fl(n,\br))$.
They are indeed listed in Figure 4.

Let us consider the case of $d=1$.
There are all four $\dyt$ diagrams,
obtained from $A_{\rm odd}$, $A_{\rm even}$,
$D_n$ and $ E_6$.
Figure 6 describes $\pi: \dy \lra \dyt$
by mapping each $\al \in \dy$ to
$\pi(\al) \in \dyt$ located exactly below $\al$.
It shows the integers
$\{m_\al \;;\; \al \in \dy\}$ and
$\{m_x \;;\; x \in \dyt\}$.
It also shows a circling $c$ on each $\dyt$,
as well as $\pi^{-1}(c)$ on $\dy$.
One checks that these are precisely all the circlings
which satisfy (\ref{sso}).
Hence Figure 6 consists of all
affine Vogan diagrams $(c,d)$ with $d=1$
and satisfy Theorem \ref{thm2}(b).

We extend Figure 6 to double Vogan diagrams,
then find their symmetric pairs in \cite[\S7]{ch}.
For example
the first affine Vogan diagram in Figure 6 is for
$L = A_{\rm odd}$. We superimpose it with Kac diagrams.
There are two such Kac diagrams, given by
$\fg = \fs \fl(n,\br)$ for $n$ even, and
$\fg = \fs \fu^*(2n)$ \cite[Fig.3]{c1}.
The resulting diagrams
are listed in Figure 4, and their symmetric pairs are
$(\fs \fl(n,\br), \fs \fl(p,\br) + \fs \fl(q,\br) + \br)$
with $p,q$ even, as well as
$(\fs \fu^*(2n), \fs \fu^*(2p) + \fs \fu^*(2q) + \br)$
\cite[p.1746]{ch}.
We repeat this method with the rest of Figure 6
and obtain Figure 4. \qed

\sk

Figures 3 and 4 are consistent with Bieliavsky's classification
of symplectic non-pseudo-Hermitian symmetric pairs
\cite[p.268-269]{bi}, except for a minor error indicated below.

\sk

\begin{remark}
{\rm (Corrigendum to \cite[p.268]{bi})
The list in the bottom of \cite[p.268]{bi} contains
$\fg =\fs \fo(p,q)$. Its invariant subalgebra
should be $\fs \fo(p-1,q-1) + \fs \fo(1,1)$
instead of $\fs \fo(p-2,q) + \br$.
This is because $\fs \fo(1,1)$ is split and
$\fs \fo(2)$ is compact.
We have made the correction in Figures 3 and 4.}
\label{bisa}
\end{remark}


\begin{figure}[h]
\includegraphics[width=5.4in,height=8.85in]{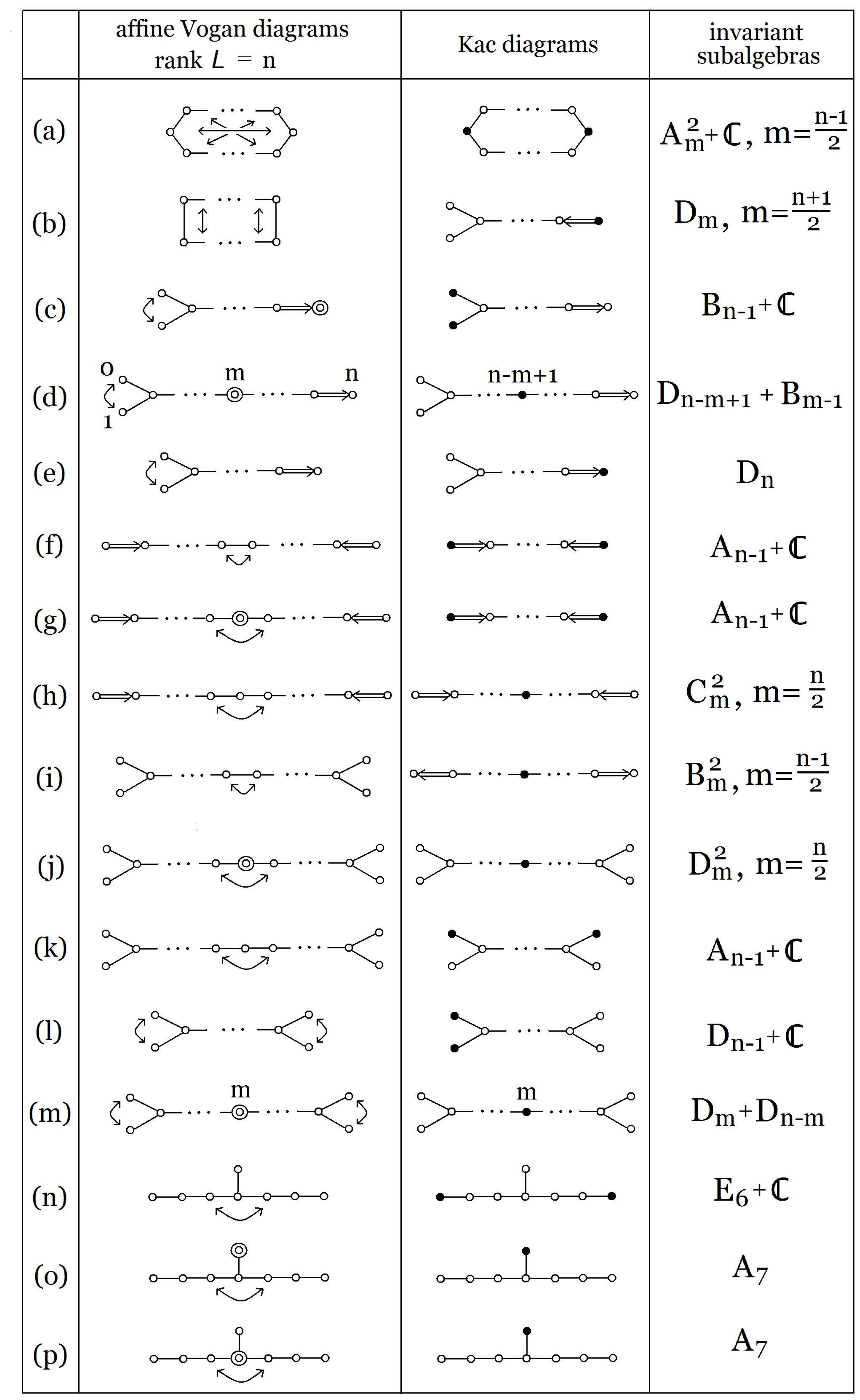}
\caption{Affine Vogan diagrams and Kac diagrams.}
\end{figure}

\begin{figure}[h]
\includegraphics[width=5in,height=2.4in]{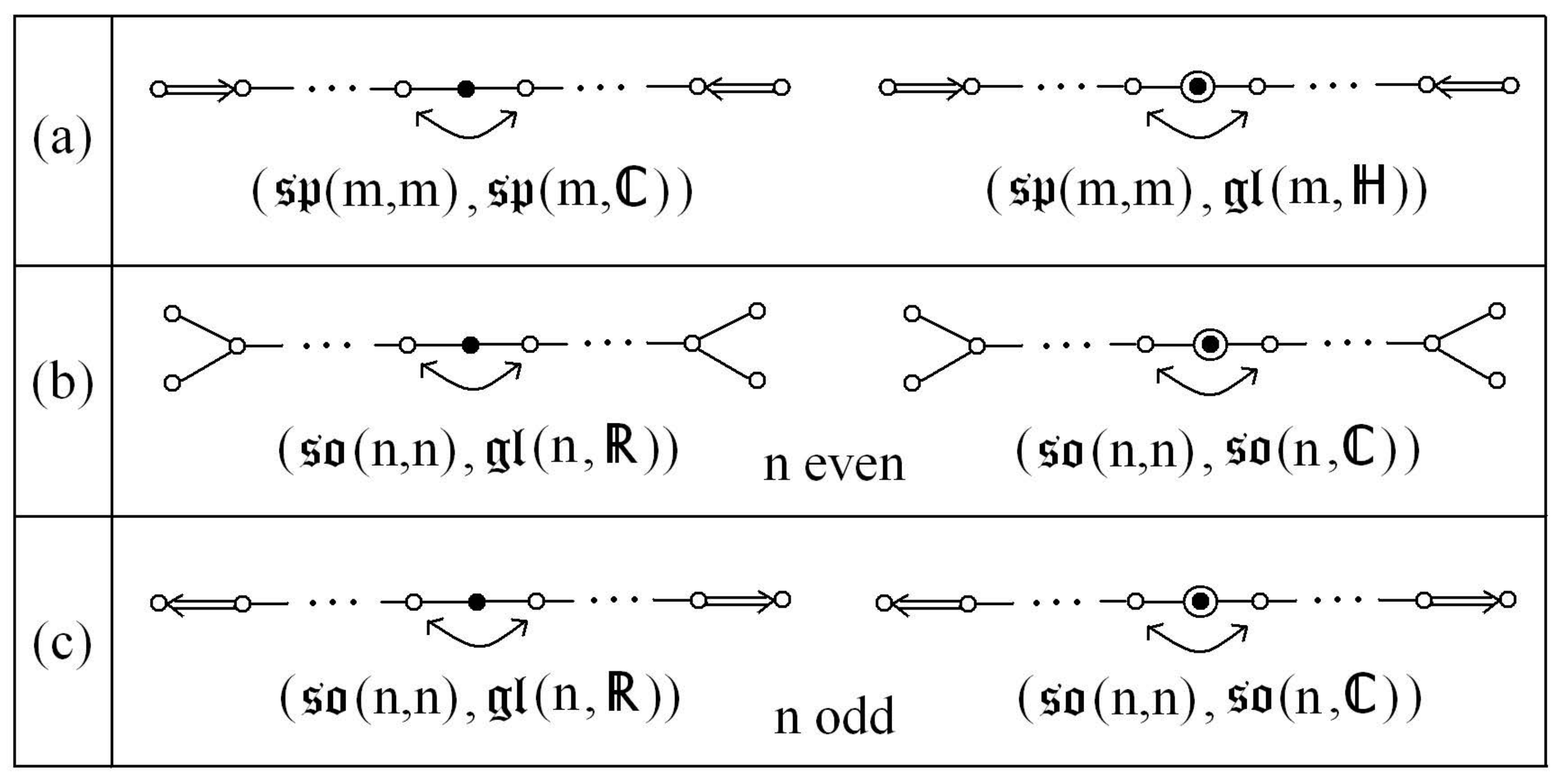}
\caption{Some associate symmetric pairs.}
\end{figure}

\begin{figure}[h]
\includegraphics[width=5in,height=3.9in]{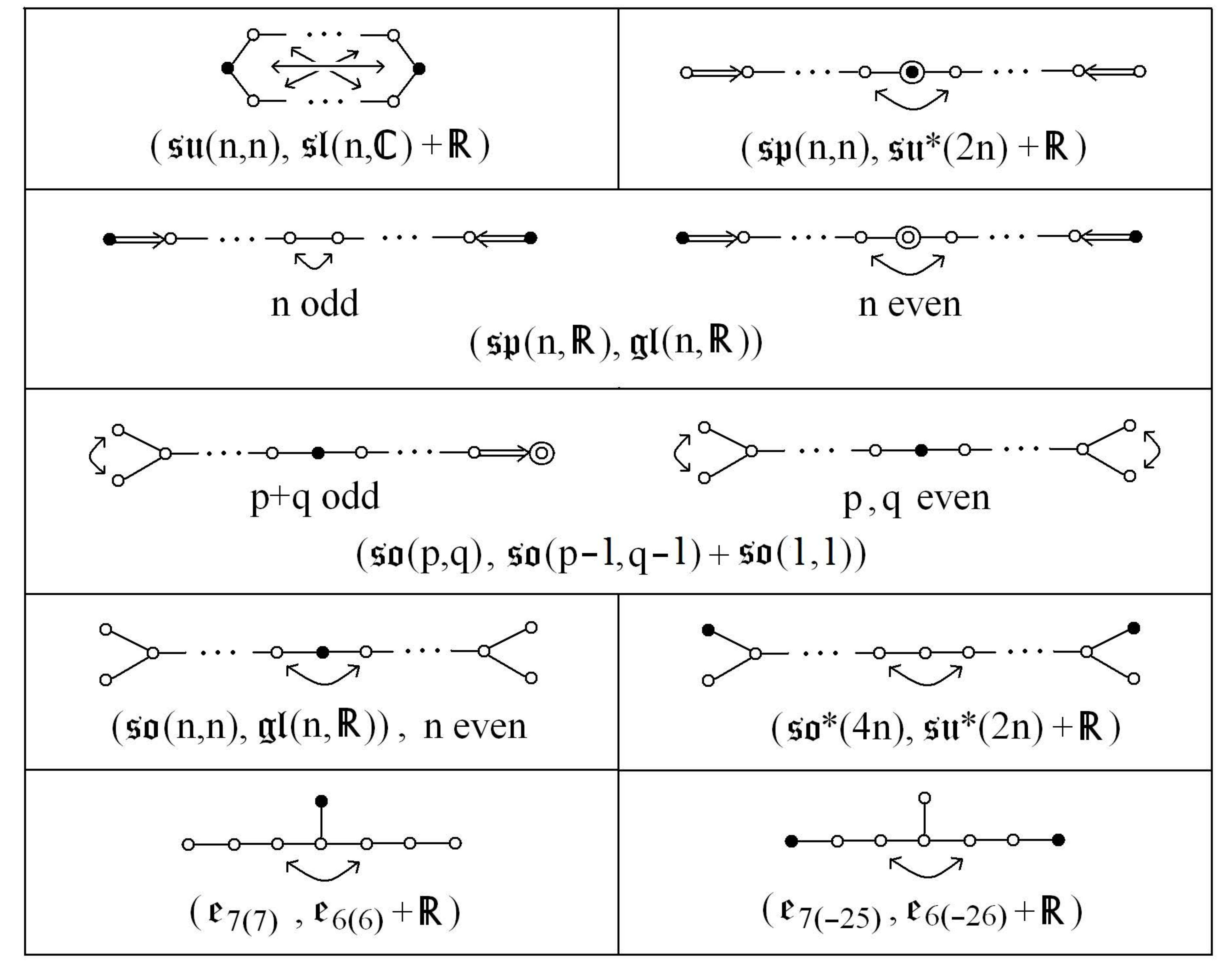}
\caption{Symplectic non-pseudo-Hermitian symmetric pairs,
$\fg$ equal rank type.}
\end{figure}

\begin{figure}[h]
\includegraphics[width=5in,height=2.76in]{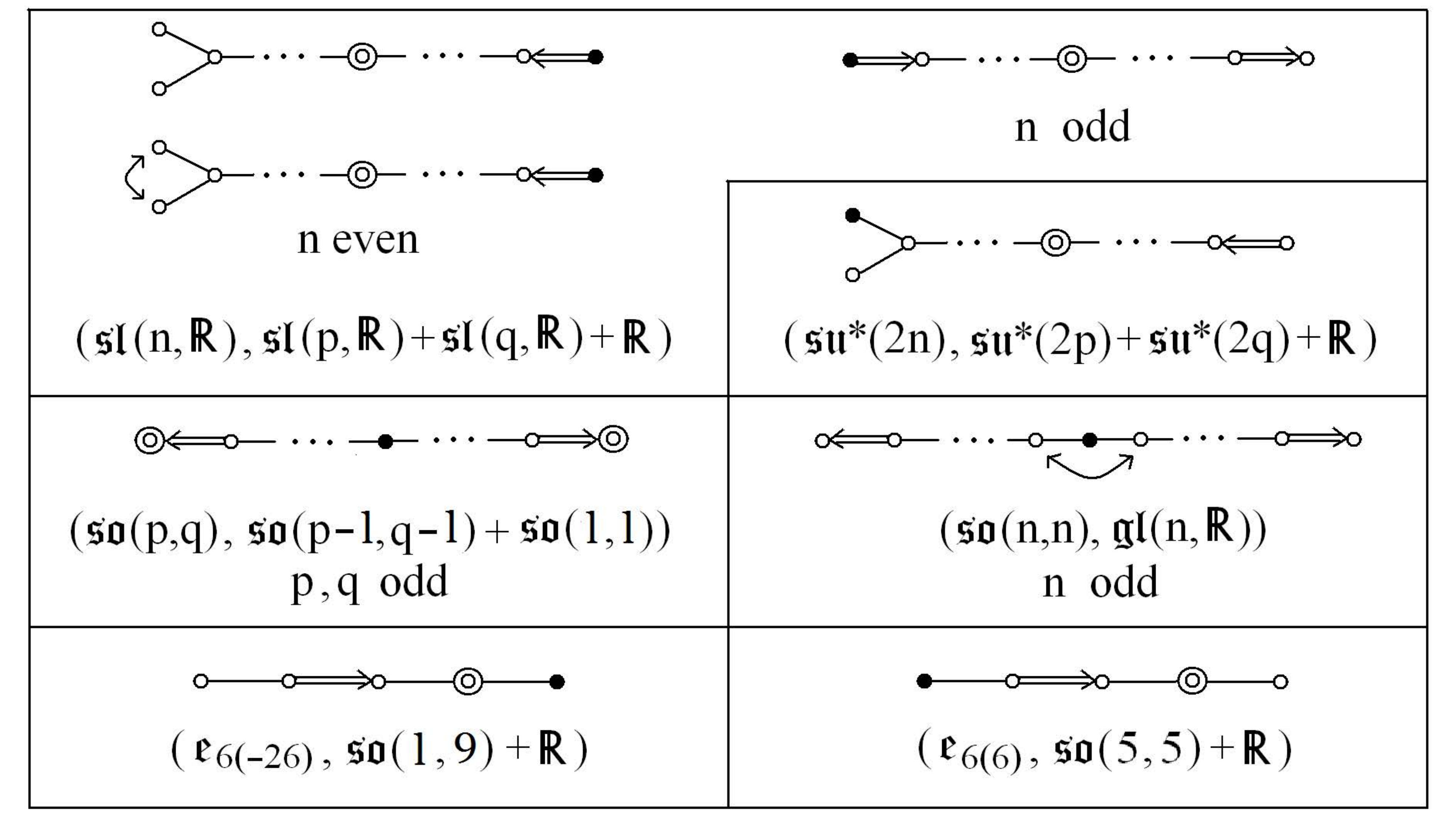}
\caption{Symplectic non-pseudo-Hermitian symmetric pairs,
$\fg$ non-equal rank type.}
\end{figure}

\begin{figure}[h]
\includegraphics[width=5.5in,height=1.32in]{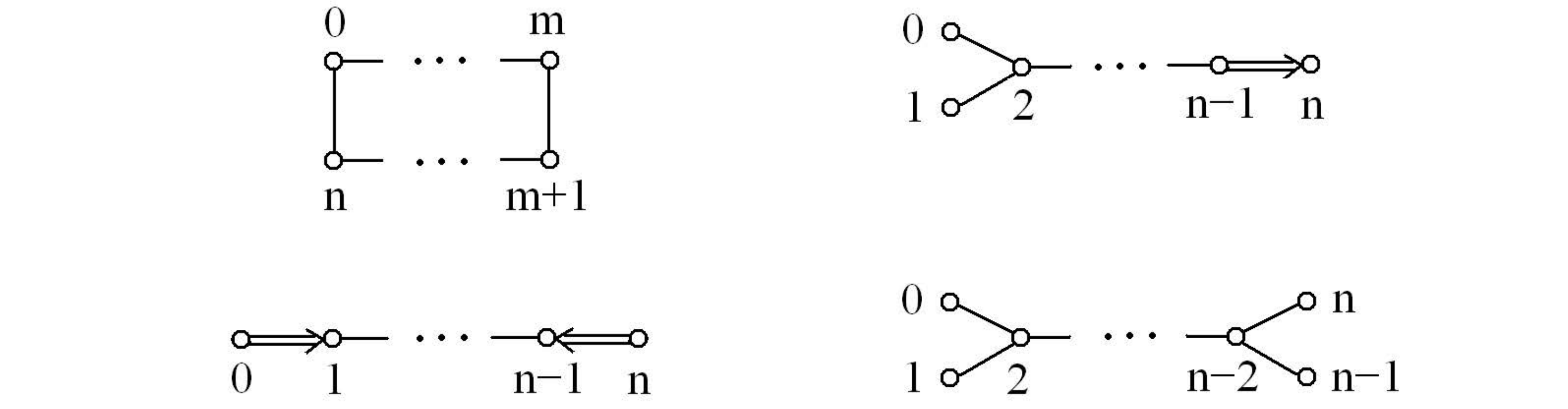}
\caption{Indices of vertices.}
\end{figure}

\begin{figure}[h]
\includegraphics[width=5.35in,height=3.6in]{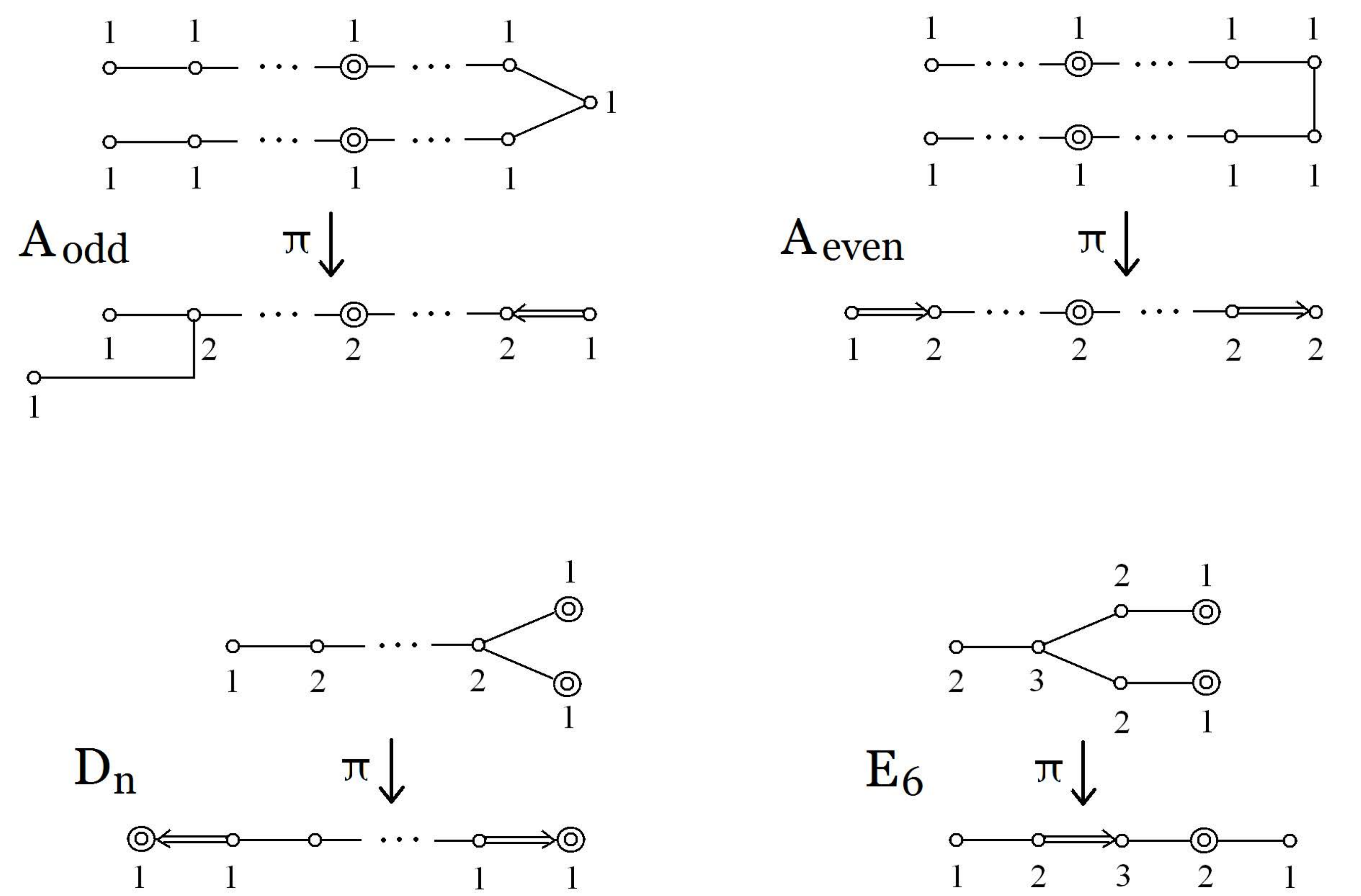}
\caption{Diagram mappings $\pi: \dy \lra \dyt$,
and circlings for Theorem \ref{thm2}(b).}
\end{figure}


\end{document}